\documentclass[a4paper,11pt]{article}
\usepackage{amsmath,amsthm,amssymb,enumitem,xcolor}

\usepackage[hang]{footmisc}
\usepackage{titling}

\usepackage[nosort,nocompress,noadjust]{cite}

\usepackage[bookmarks=false,hyperfootnotes=false,colorlinks,
    linkcolor={red!60!black},
    citecolor={blue!50!black},
    urlcolor={blue!80!black}]{hyperref}

\renewcommand{\eqref}[1]{\hyperref[#1]{(\ref{#1})}}

\usepackage{todonotes}

\newlist{enumlist}{enumerate}{2}
\setlist[enumlist,1]{labelindent=0cm,label=\arabic*.,ref=\arabic*,labelwidth=2.5ex,labelsep=0.5ex,leftmargin=3ex,align=left,topsep=0.5ex,itemsep=1ex,parsep=1ex}
\setlist[enumlist,2]{labelindent=0cm,label=\theenumlisti.\arabic*.,ref=\arabic*,labelwidth=5ex,labelsep=0.5ex,leftmargin=5.5ex,align=left,topsep=0.5ex,itemsep=1ex,parsep=1ex}

\newlist{itemlist}{itemize}{1}
\setlist[itemlist]{labelindent=0cm,label=$\bullet$,labelwidth=2.5ex,labelsep=0.5ex,leftmargin=3ex,align=left,topsep=0.5ex,itemsep=1ex,parsep=1ex}

\numberwithin{equation}{section}

{\theoremstyle{definition}\newtheorem{definition}{Definition}[section]
\newtheorem*{definition*}{Definition}

\newtheorem*{example*}{Example}
\newtheorem*{examples*}{Examples}}

\newtheorem{proposition}[definition]{Proposition}
\newtheorem{lemma}[definition]{Lemma}
\newtheorem{theorem}[definition]{Theorem}

\newtheorem{letterthm}{Theorem}

{\theoremstyle{definition}}

\newcommand{\C}{\mathbb{C}}

\newcommand{\al}{\alpha}
\newcommand{\be}{\beta}

\newcommand{\recht}{\rightarrow}
\newcommand{\Z}{\mathbb{Z}}
\newcommand{\vphi}{\varphi}

\newcommand{\id}{\mathord{\text{\rm id}}}

\newcommand{\N}{\mathbb{N}}

\newcommand{\Om}{\Omega}

\newcommand{\R}{\mathbb{R}}
\newcommand{\F}{\mathbb{F}}

\newcommand{\actson}{\curvearrowright}

\newcommand{\cU}{\mathcal{U}}
\newcommand{\Ker}{\operatorname{Ker}}

\newcommand{\lspan}{\operatorname{span}}

\newcommand{\cR}{\mathcal{R}}

\newcommand{\SL}{\operatorname{SL}}

\newcommand{\cRbar}{\overline{\cR}}

\title{KMS spectra for group actions on compact spaces}
\author{ by Johannes Christensen\thanks{ KU Leuven, Department of Mathematics, Leuven (Belgium), johannes.christensen@kuleuven.be \newline Supported by a DFF-International Postdoctoral Grant.} \ and Stefaan Vaes\thanks{KU Leuven, Department of Mathematics, Leuven (Belgium), stefaan.vaes@kuleuven.be \newline
Supported by FWO research project G090420N of the Research Foundation Flanders and by long term structural
funding – Methusalem grant of the Flemish Government.}}
\date{\today}

\thanksmarkseries{arabic}

\begin{document}

\maketitle

\begin{abstract}
\noindent Given a topologically free action of a countable group $G$ on a compact metric space $X$, there is a canonical correspondence between continuous $1$-cocycles for this group action and diagonal $1$-parameter groups of automorphisms of the reduced crossed product C$^*$-algebra. The KMS spectrum is defined as the set of inverse temperatures for which there exists a KMS state. We prove that the possible KMS spectra depend heavily on the nature of the acting group $G$. For groups of subexponential growth, we prove that the only possible KMS spectra are $\{0\}$, $[0,+\infty)$, $(-\infty,0]$ and $\R$. For certain wreath product groups, which are amenable and of exponential growth, we prove that any closed subset of $\R$ containing zero arises as KMS spectrum. Finally, for certain nonamenable groups including the free group with infinitely many generators, we prove that any closed subset may arise. Besides uncovering a surprising relation between geometric group theoretic properties and KMS spectra, our results provide two simple C$^*$-algebras with the following universality property: any closed subset (containing, resp. not containing zero) arises as the KMS spectrum of a $1$-parameter group of automorphisms of this C$^*$-algebra.
\end{abstract}

\section{Introduction}
The crossed product construction for groups acting by homeomorphisms on compact spaces has nurtured a mutually beneficial interplay between dynamical systems and operator algebras, which has made crossed products one of the cornerstones in the theory of C$^*$-algebras. For a countable discrete group $G$ it is a well established fact that amenability of the group is connected to the existence of certain states on the crossed product: the group $G$ is amenable if and only if all crossed products of $G$ admit a tracial state. This connection hinges on the fact that both amenability of the acting group $G$ and the existence of tracial states for the crossed product of an action of $G$ on a compact space $X$ are closely connected to the existence of a $G$-invariant probability measure on $X$. One way to formulate this, is that amenability of the group $G$ is completely reflected in the possible behavior of the KMS spectra of the trivial $1$-parameter group on the crossed products by $G$. In this article, we reveal a similar but much subtler phenomenon for non-trivial $1$-parameter groups. We
prove that the possible KMS spectra depend as follows on the acting group $G$: when $G$ has
polynomial growth, only the subsets $\{0\}$, $[0,+\infty)$, $(-\infty,0]$ and $\R$ arise as KMS spectrum; for general amenable groups, all closed subsets of $\R$ containing zero can arise and are concretely realized for certain wreath product groups;
while for arbitrary countable groups, any closed subset of $\R$ may appear and is concretely realized for the free group with infinitely many generators.

KMS states on C$^*$-algebras have been widely studied. First, for natural families of C$^*$-algebras and $1$-parameter groups of
automorphisms, the KMS spectrum has been determined and for each admissible inverse temperature, the simplex of KMS states computed, often exhibiting remarkable phase transition phenomena when the temperature increases. See e.g. \cite{OP, aHLRS, T1, KR, ALN20}. Second, the question which sets may arise as KMS
spectrum has been addressed early on, in the seminal paper \cite{BEH}, and it has recently seen great progress in the case of simple AF-algebras \cite{T3}. In this article, we focus on crossed product C$^*$-algebras given by group actions on compact spaces. The study of KMS states for C$^{*}$-algebras arising from invertible dynamical systems has only recently come into focus in \cite{CT}, where it was proven that KMS spectra are rigid for diagonal actions of $\Z$.

To state our results, we need some more terminology and notation. When $A = C(X) \rtimes_r G$ is the reduced crossed product of a countable group $G$ acting on a
compact metric space $X$, we specifically consider the diagonal $1$-parameter groups $\alpha$, i.e.
acting as the identity on $C(X)$. Every continuous $1$-cocycle $\Omega : G \times X \to \R$ canonically
leads to such a $1$-parameter group given by $\alpha^{\Omega}_{t}(U_{g}) = U_{g}e^{-it\Omega(g, \cdot)}$ for $t\in \R$ and $g\in G$, where $\{U_{g}\}_{g\in G}$ are the unitaries in $C(X)\rtimes_{r} G$ implementing the action of $G$ on $X$. If the action $G \curvearrowright X$ is topologically
free, all diagonal $1$-parameter groups of $A$ arise in this way. Similar to the tracial state case, the existence of a $\beta$-KMS state for such a $1$-parameter group is equivalent to the existence of a $\beta$-conformal probability measure on $X$, i.e.\ a probability measure $\mu$ on $X$ such that
$G \curvearrowright (X,\mu)$ is nonsingular with Radon-Nikodym derivative $\mathrm{d} g^{-1}\mu / \mathrm{d} \mu = e^{\beta \cdot \Omega(g, \cdot)}$ for $g\in G$. In this article we study the possible behaviors of the \emph{KMS spectra} arising from such $1$-cocycles $\Omega$, namely the subsets of $\R$ defined as
\begin{align*}
&\{ \beta \in \R \; | \: \text{ there exists a $\beta$-KMS state for $\alpha^{\Omega}$ on $C(X)\rtimes_{r} G$} \} \\
=& \{\beta \in \R \; | \: \text{ there exists a $\beta$-conformal probability measure on $X$}\} \; .
\end{align*}
We prove that the range of possible KMS spectra depends in a surprising way on the structure of the
group $G$. First we exploit the connection to ergodic theory to prove that for groups of subexponential growth these KMS spectra are extremely rigid.

\begin{letterthm}\label{theorem.A}
Let $G$ be a countable group such that every finitely generated subgroup of $G$ has subexponential growth. Let $G \actson X$ be an action by homeomorphisms of a compact space $X$. Let $\Om : G \times X \recht \R$ be any continuous $1$-cocycle. Then the KMS spectrum must have one of the following forms: $\{0\}$, $[0,+\infty)$, $(-\infty,0]$ or $\R$.
\end{letterthm}

Considering Theorem \ref{theorem.A} it would be natural to hypothesize that the KMS spectra for the C$^*$-dynamical systems that we consider are in general rigid for all groups $G$. This however turns out to be far from the case, even when staying within the class of amenable groups, as can be seen from the following main theorem of our article.

\begin{letterthm}\label{thm.all-closed-sets}
\begin{enumlist}
\item \label{thma1} Any wreath product group $\Gamma = \Lambda \wr \Z = \Lambda^{(\Z)} \rtimes \Z$ where $\Lambda$ is an infinite direct sum of finite groups admits a minimal topologically free action $\Gamma \actson X$ with the following universality property: for every closed subset $K \subset \R$ with $0 \in K$, there exists a continuous $1$-cocycle $\Gamma \times X \recht \R$ with KMS spectrum $K$ and with a unique $\beta$-KMS state for every $\beta \in K$.

\item\label{thma2} The free group $\F_{\infty}$ admits a free and minimal action $\F_\infty \actson X$ with the following universality property: for every closed subset $K \subset \R$ with $0 \not\in K$, there exists a continuous $1$-cocycle $\F_\infty \times X \recht \R$ with KMS spectrum $K$ and with a unique $\beta$-KMS state for every $\beta \in K$.
\end{enumlist}
\end{letterthm}

We prove part~1 in Theorem \ref{thm.all-closed-sets-amenable} and we prove part~2 in Section \ref{sec.all-closed-sets-nonamenable} below. We will see that these actions of $\Gamma$, resp.\ $\F_\infty$ are entirely explicit.
Also note that for arbitrary group actions, the KMS spectrum is always a closed set and that $0$ belongs to the KMS spectrum if and only if the action admits an invariant probability measure, which is always the case if the acting group is amenable and which is independent of the $1$-cocycle. So, the universality properties in Theorem \ref{thm.all-closed-sets} exhibit the most general possible behavior.

Besides demonstrating an overlooked connection between the structure of groups and KMS spectra, Theorem \ref{thm.all-closed-sets} also improves state of the art on constructions of KMS spectra in general. In a seminal paper \cite{BEH}, Bratteli, Elliott and Herman construct for a given closed set $K \subset \R$ a simple C$^*$-algebra and a $1$-parameter group with associated KMS spectrum $K$. The strategy of proof in \cite{BEH} is to use a closed set $K \subset \R$ to construct a dimension group, and then use classification results to construct a simple C$^*$-algebra and a $1$-parameter group with KMS spectrum $K$. This idea has since been used several times with variations to build interesting C$^*$-dynamical systems, see e.g. \cite{BEK, T3}. Despite the decade-long interest in the problem, our construction of two fixed simple C$^*$-algebras on which all KMS spectra can be realized is the first of its kind. Furthermore, our strategy of proof takes a different approach to the problem, by building the $1$-parameter groups ``by hand'', which has as a further advantage that our article is completely self-contained, and the proof of our main theorem does not require any results from classification theory of simple C$^*$-algebras.

\smallskip
{\center
\emph{Acknowledgements.} The proof of Theorem \ref{theorem.A} uses unpublished ideas developed by Klaus Thomsen and the first named author while working on \cite{CT}, which handles the special case $G=\mathbb{Z}$. We are grateful to Klaus Thomsen for allowing us to include them in this article, and for discussions leading to the results in Section \ref{section5}.
}

\section{Notations and preliminaries}

In this article $X$ and $Y$ will always denote compact metric spaces and $G$ and $\Lambda$ will always denote countable discrete groups. We denote the unit of any group by $e$. To keep this paper self-contained, we gather in this section several known and elementary results on
crossed product C$^{*}$-algebras, KMS states and their relation to conformal measures.

\subsection{Crossed products and KMS states}
Let $G$ be a countable group and let $X$ be a compact metric space, or equivalently a second countable compact Hausdorff space. An action of $G$ on $X$ is a representation of $G$ by homeomorphisms of $X$, i.e.\ it is a family $\{\phi_g\}_{g\in G}$ of homeomorphisms of $X$ satisfying $\phi_{gh} = \phi_g \circ \phi_h$ for all $g,h \in G$ and $\phi_{e}=\id_{X}$. We write this as $G\actson X$, and we often write $g \cdot x$ for $\phi_{g}(x)$ when the action is clear from context. Recall that the isotropy group at a point $x\in X$ is the subgroup $\{ g\in G \; : \; g\cdot x=x\}$ of $G$. The action $G\actson X$ is \emph{minimal} when $\{ g\cdot x \; :\; g\in G\}$ is dense in $X$ for all $x\in X$ and it is \emph{topologically free} if the set of elements in $X$ with trivial isotropy group is dense in $X$. For a group $G$, we denote the group $\bigoplus_{\N} G$ by $G^{(\N)}$, and for a space $X$ we let $X^{\N}$ denote the product space $\prod_{\N} X$.

Recall the definition of the \emph{reduced crossed product} for a countable group $G$ that acts on the compact metric space $X$. Take a unital embedding $C(X) \subset B(H)$ for some Hilbert space $H$, and let $\delta_{g, \xi} \in l^{2}(G, H)$ denote the map with $\delta_{g, \xi}(h)=0$ for $h\neq g$ and $\delta_{g, \xi}(g)=\xi$. Define a representation of $C(X)$ and $G$ on $ l^{2}(G, H)$ by
$$
f \delta_{g,\xi} = \delta_{g, f\circ \phi_{g} \xi} \quad \text{ and } \quad U_{h} \delta_{g,\xi} = \delta_{hg,\xi} \quad \text{ for $f\in C(X)$ and $h\in G$.}
$$
The \emph{reduced crossed product} $C(X)\rtimes_{r} G$ is the C$^*$-subalgebra of $B(l^{2}(G,H))$ generated by $C(X)$ and $\{U_{g}\}_{g\in G}$. We call the copy of $C(X)$ in $C(X)\rtimes_{r} G$ the diagonal. By construction
\begin{equation}\label{eq.unitary}
U_{g} f U_{g}^{*} = f\circ \phi_{g}^{-1} = f\circ \phi_{g^{-1}} \quad \text{ for all $g\in G$ and all $f\in C(X)$} \ ,
\end{equation}
and $C(X)\rtimes_{r} G$ contains
\begin{equation}\label{eq.dense-span}
\lspan \{ f U_{g} \; | \; f\in C(X) \text{ and } g\in G \}
\end{equation}
as a dense $*$-subalgebra.  There exists a conditional expectation $P : C(X)\rtimes_{r} G \to C(X)$ which is determined by the condition that $P(fU_{g})=0$ for $g\neq e$ and $f\in C(X)$. If $G$ acts on $X$ by a family $\{\phi_g\}_{g\in G}$ of homeomorphisms we call a map $\Omega: G \times X \to \mathbb{R}$ a \emph{continuous $1$-cocycle} if it is continuous and
\begin{equation} \label{eq.cocycle}
\Omega(g, \phi_{h}(x)) +\Omega(h,x) = \Omega(gh, x) \quad \text{ for all $g,h \in G$ and $x\in X$.}
\end{equation}
The continuous $1$-cocycles give rise to the continuous $1$-parameter groups we will investigate in this article. The following result is well known. For completeness, we provide a proof.

\begin{lemma} \label{lem.diagonalfixing}
Assume $G\actson X$ and $\Omega: G \times X \to \mathbb{R}$ is a continuous $1$-cocycle. There exists a continuous $1$-parameter group $\{\alpha_{t}^{\Omega}\}_{t\in \R}$ on $C(X) \rtimes_{r}G$ such that
$$
\alpha^{\Omega}_{t} (fU_{g}) = fU_{g} e^{-i t \Omega(g, \cdot)}
$$
for all $f\in C(X)$ and all $g\in G$ and $t\in \R$. In particular, $\alpha^{\Omega}_{t}(f)=f$ for all $f\in C(X)$ and all $t\in \R$.

If the action $G \actson X$ is topologically free, any continuous $1$-parameter group on $C(X) \rtimes_{r} G$ that fixes the diagonal $C(X)$ arises from a continuous $1$-cocycle $\Omega: G \times X \to \mathbb{R}$.
\end{lemma}

\begin{proof}
For a continuous $1$-cocycle $\Omega: G \times X \to \mathbb{R}$ we can define a $1$-parameter group of unitaries $\{ V_{t} \}_{t\in \R}$ on $l^{2}(G, H)$ by setting
$$
V_{t} \delta_{g, \xi} = \delta_{g, \exp(-it\Omega(g, \cdot))\xi} \quad \text{ for $g\in G$, $t\in \R$ and $\xi \in H$,}
$$
and setting $\alpha^{\Omega}_{t}(A)=V_{t} AV_{t}^{*}$ for $A\in C(X)\rtimes_{r} G$ then defines the desired continuous $1$-parameter group. Assume now that $G$ acts topologically freely and $\alpha$ fixes $C(X)$. Let $g, h\in G$ with $g\neq h$ and let $x\in X$ have trivial isotropy group. Write $1=f_{1}+f_{2}$ with $f_{1}, f_{2} \in C(X)$ such that $f_{1}=0$ on a neighborhood $V$ of $\phi_{h}(x)$ and $f_{2}=0$ on a neighborhood $U$ of $\phi_{g}(x)$. Pick $\psi_{1} \in C_{c}(V)$ with $\psi_{1}(\phi_{h}(x))=1$ and $\psi_{2} \in C_{c}(U)$ with $\psi_{2}(\phi_{g}(x))=1$, then
$$
P(U_{h}^{*}\alpha_{t}(f_{1}U_{g}))(x) = P(\psi_{1}\circ \phi_{h}U_{h}^{*}\alpha_{t}(f_{1}U_{g}))(x)
= P(U_{h}^{*}\alpha_{t}(\psi_{1}f_{1}U_{g}))(x) =0
$$
and likewise one can argue that $P(U_{h}^{*}\alpha_{t}(f_{2}U_{g}))(x)=0$ by multiplying $\psi_{2}\circ \phi_{g}$ on the right. By topological freeness $P(U_{h}^{*}\alpha_{t}(U_{g}))=0$. If $P_{g}$ is the orthogonal projection onto $\{\delta_{g, \xi} : \xi \in H\}$, then $P(A)=\sum_{g\in G} P_{g}AP_{g}$, so $P(U_{h}^{*}\alpha_{t}(U_{g}))=0$ for $h\neq g$ implies that $U_{g}^{*}\alpha_{t}(U_{g})$ commutes with the projections $P_{h}$, $h\in G$. Hence $U_{g}^{*}\alpha_{t}(U_{g}) = P(U_{g}^{*}\alpha_{t}(U_{g})) \in C(X)$ is a unitary for each $t\in \R$. Stone's Theorem implies that the norm-continuous $1$-parameter group of unitaries $\{U_{g}^{*}\alpha_{t}(U_{g})\}_{t\in \R}$ is of the form $\{ e^{it\Omega(g, \cdot)} \}_{t\in \R}$ for some continuous real valued function $\Omega(g, \cdot) \in C(X)$. Since $\al$ is a $1$-parameter group of automorphisms, the functions $\{\Omega(g,\cdot) \}_{g\in G} $ defines a continuous $1$-cocycle, which proves the Lemma.
\end{proof}

The aim of this article is to investigate KMS states on the crossed products $C(X)\rtimes_{r} G$ for the $1$-parameter groups $\alpha^{\Omega}$ arising from continuous $1$-cocycles $\Omega: G \times X \to \mathbb{R}$. We refer the reader to \cite{BR} for a thorough introduction to KMS states and its relevance in quantum statistical mechanics. Recall that a \emph{$\beta$-KMS state} for a $\beta \in \R$ and the continuous $1$-parameter group $\alpha^{\Omega}$ on $C(X) \rtimes_{r} G$ is a state $\omega$ on $C(X) \rtimes_{r} G$ satisfying that
\begin{equation} \label{eq.KMS-condition}
\omega(AB) = \omega(B \alpha^{\Omega}_{i\beta}(A))
\end{equation}
for all $A$, $B$ in a norm dense, $\alpha^{\Omega}$-invariant $*$-subalgebra of $C(X) \rtimes_{r} G$ consisting of $\alpha^{\Omega}$ analytic elements. Equivalently, $\omega$ is a KMS state if \eqref{eq.KMS-condition} holds true for all elements $A$ and $B$ in the set \eqref{eq.dense-span}, c.f. Proposition 5.3.7 in \cite{BR}. Hence a state $\omega$ on $C(X) \rtimes_{r} G$ is a $\beta$-KMS state for $\alpha^{\Omega}$ if and only if
\begin{equation} \label{eq.KMS-reduction}
\omega(fU_{g} qU_{h}) = \omega( qU_{h} fU_{g} e^{\beta \cdot \Omega(g, \cdot)}) \quad \text{ for all $g,h\in G$ and $f,q \in C(X)$.}
\end{equation}
The question of existence of $\beta$-KMS states for different values of $\beta \in \R$ will be important in this article, so we will often work with the \emph{KMS spectrum} for $\alpha^{\Omega}$, which is the subset of $\mathbb{R}$ given as
$$
\{ \beta \in \R \mid \text{there exists a $\beta$-KMS state for $\alpha^{\Omega}$ on $C(X) \rtimes_{r} G$}  \} \ .
$$
It follows from the general theory on KMS states that a KMS spectrum is always closed, c.f. Proposition 5.3.23 in \cite{BR}.

\subsection{Conformal measures}
Assume that $G$ is a countable group, $X$ is a compact metric space and $G \actson X$. The Riesz representation theorem implies that any state $\omega$ on $C(X) \rtimes_{r} G$ defines a unique Borel probability measure on $X$ by restricting $\omega$ to the diagonal $C(X)$. In this section we will argue, that for the scope of this article, it suffices to describe the measures on $X$ arising from KMS states to obtain a description of the KMS states. The results we present in the following are standard observations if one describes the crossed product C$^*$-algebra as a groupoid C$^*$-algebra and use general results of Neshveyev and Renault \cite{N, R}, but for readability we give elementary proofs of these facts.

\begin{lemma} \label{lemma.states-conformal}
Assume $G\actson X$ and $\Omega: G\times X \to \mathbb{R}$ is a continuous $1$-cocycle. Let $\omega$ be a $\beta$-KMS state for $\alpha^{\Omega}$ on $C(X) \rtimes_{r} G$ and let $\mu_{\omega}$ be the probability measure on $X$ arising from restricting $\omega$ to $C(X)$. Then
\begin{equation}\label{eq.conformal-measure}
(g^{-1}\mu_{\omega})(B) :=\mu_{\omega}(\phi_{g}(B))=\int_{B} e^{\beta \cdot \Omega(g, x)} \ \mathrm{d} \mu_{\omega}(x) \quad \text{ for any Borel set $B\subset X$ and all $g\in G$.}
\end{equation}

Conversely, if $\mu$ is a probability measure on $X$ which satisfies \eqref{eq.conformal-measure} then
$$
A \to \int_{X} P(A) \ \mathrm{d}\mu
$$
defines a $\beta$-KMS state for $\alpha^{\Omega}$ on $C(X) \rtimes_{r} G$.
\end{lemma}

Before turning to the proof, let us note that a probability measure $m$ on $X$ satisfies \eqref{eq.conformal-measure} if and only if
\begin{equation}\label{eq.conformint}
\int_{X} f \circ \phi_{g} \; e^{\beta \cdot \Omega(g, \cdot )} \ \mathrm{d} m = \int_{X} f  \ \mathrm{d} m  \quad \text{ for all $g\in G$ and $f\in C(X)$.}
\end{equation}

\begin{proof}
Let $\omega$ be a $\beta$-KMS state for $\alpha^{\Omega}$ on $C(X) \rtimes_{r} G$. For $f\in C(X)$ and $g\in G$ the KMS condition \eqref{eq.KMS-reduction} implies that
$$
\omega(f)=\omega \left( U_{g} f\circ\phi_{g} U_{g}^{*}\right)
= \omega \left( f\circ\phi_{g} U_{g^{-1}}  U_{g} e^{\beta \cdot \Omega(g, \cdot)}\right)
=\omega \big( f\circ\phi_{g} e^{\beta \cdot \Omega(g, \cdot)} \big) \
$$
which proves \eqref{eq.conformal-measure}. Assume conversely that $\mu$ is a measure satisfying \eqref{eq.conformal-measure}, and set
$$
\omega(A)= \int_{X} P(A) \ \mathrm{d}\mu \; \text{ for $A\in C(X) \rtimes_{r} G$.}
$$
Since $P(fU_{g}qU_{h})=0$ for any $f,q\in C(X)$ and $g,h \in G$ with $h\neq g^{-1}$, it suffices by \eqref{eq.KMS-reduction} to check that
$$
\omega( f \; q\circ \phi_{g^{-1}} ) = \omega(q \; f\circ \phi_{g} \; e^{\beta \cdot \Omega(g, \cdot )} )
$$
for all $f,q\in C(X)$ and $g\in G$, which is a straightforward consequence of \eqref{eq.conformint}.
\end{proof}

In a more general setup than the one considered in this article, measures satisfying \eqref{eq.conformal-measure} have been studied for decades in the field of dynamical systems. We follow the terminology introduced in dynamical systems \cite{DU}, and call a probability measure on $X$ that satisfies \eqref{eq.conformal-measure} an \emph{$e^{\beta \cdot \Omega}$-conformal measure} for the action $G \actson X$. If there can be no confusion which action $G\actson X$ and which continuous $1$-cocycle $\Omega$ we are referring to, we will simply call such a measure \emph{$\beta$-conformal}. It follows from Lemma \ref{lemma.states-conformal} that there exists a $\beta$-KMS state for $\alpha^{\Omega}$ on $C(X) \rtimes_{r} G$ if and only if there exists a $\beta$-conformal measure on $X$. So determining the KMS spectrum for $\alpha^{\Omega}$ is equivalent to determining for which values of $\beta \in \R$ there exists a $\beta$-conformal measure on $X$.

The surjective map $\omega \to \mu_{\omega}$ between $\beta$-KMS states for $\alpha^{\Omega}$ and $\beta$-conformal measures on $X$ described in Lemma \ref{lemma.states-conformal} is in general not injective, see e.g. Example 3.4 in \cite{CT}, but there is a simple sufficient condition that ensures that this map is a bijection. To describe this condition, we call a measure $\mu$ on $X$ \emph{essentially free} when $\mu$ is concentrated on points of $X$ with trivial isotropy group. When $\mu$ is conformal this corresponds precisely to essential freeness of the action $G \actson (X, \mu)$.

\begin{lemma} \label{lemma.trivial-isotropy}
Let $G\actson X$, let $\Omega: G \times X \to \mathbb{R}$ be a continuous $1$-cocycle and let $\omega$ be a $\beta$-KMS state for $\alpha^{\Omega}$ on $C(X)\rtimes_{r} G$. If $\mu_{\omega}$ is essentially free then $\omega$ is on the form
$$
\omega(A) = \int_{X} P(A) \ \mathrm{d} \mu_{\omega} \quad \text{ for all $A\in C(X) \rtimes_{r} G$.}
$$
\end{lemma}

\begin{proof}
This proof follows exactly as the proof of Theorem 3.10 in \cite{CT}.
\end{proof}

When all conformal measures are essentially free, Lemma \ref{lemma.trivial-isotropy} implies that we can describe all KMS states on $C(X)\rtimes_{r} G$ by describing the conformal measures.

\begin{lemma} \label{lemma.cstaralgeba}
Let $G\actson X$, let $\Omega: G \times X \to \mathbb{R}$ be a continuous $1$-cocycle and let $\beta \in \R$. Assume all $\beta$-conformal measures are essentially free. We have an affine homeomorphism between
\begin{enumerate}
\item the $\beta$-conformal measures on $X$, and
\item the $\beta$-KMS states for $\alpha^{\Omega}$ on $C(X)\rtimes_{r} G$.
\end{enumerate}
The homeomorphism is given by $\mu \to \int_{X}P(\cdot) \; \mathrm{d}\mu$.
\end{lemma}

In this article, we will only focus on two questions concerning KMS states. First, we will focus on the question of existence of $\beta$-KMS states, which is equivalent to the question of existence of $\beta$-conformal measures. Secondly, we will focus on systems where all $\beta$-conformal measures are essentially free, which by Lemma \ref{lemma.cstaralgeba} implies that describing the $\beta$-conformal measures is equivalent to describing the $\beta$-KMS states. We will therefore restrict our attention to describing conformal measures for the rest of the article. The following observation makes it substantially easier to verify if a measure is conformal. The result follows immediately because $1$-cocycles are completely determined by their values on a
generating set. The proof is left as an exercise.

\begin{lemma} \label{lem.genconformal}
Let $G$ be a group generated by a set $S$. Assume that $G\actson X$ and that $\Omega: G\times X\to \R$ is a continuous $1$-cocycle. A probability measure $m$ on $X$ is $\beta$-conformal if and only if it satisfies \eqref{eq.conformal-measure} for $g\in S$.
\end{lemma}

If $G\actson X$ and $H: X\to \R$ is a continuous function, we can define a continuous $1$-cocycle $\Omega: G\times X\to \R$ by setting
\begin{equation}\label{eq.coboundary}
\Omega(g,x):=H(\phi_{g}(x))-H(x) \quad \text{  for $g\in G$ and $x\in X$} \; ,
\end{equation}
and we call such a continuous $1$-cocycle a \emph{coboundary}. The set of continuous $1$-cocycles $\Omega: G\times X\to \R$ is an abelian group with addition as operation, and we call two continuous $1$-cocycles \emph{cohomologous} if their difference is a coboundary. The following result on the relation between the $\beta$-conformal measures for cohomologous $1$-cocycles is standard to prove, so we leave its verification to the reader.

\begin{lemma}\label{lemma.cohomologous}
Let $G\actson X$ and let $\Omega: G\times X\to \R$ be a continuous $1$-cocycle. Assume that $H:X\to \mathbb{R}$ is a continuous function, and define a continuous $1$-cocycle $\tilde{\Omega}:G\times X\to \R$ by
$$
\tilde{\Omega}(g, \cdot):=\Omega(g,\cdot)+H\circ \phi_{g}(\cdot) -H(\cdot) \quad \text{ for all $g\in G$}.
$$
Let $\beta \in \R$. The $e^{\beta \cdot \Omega}$-conformal measures are in bijective correspondence with the $e^{\beta \cdot \tilde{\Omega}}$-conformal measures via the map
$$
\mathrm{d} m \mapsto \frac{e^{\beta H}\mathrm{d} m}{\int_{X} e^{\beta H}\mathrm{d} m} \ .
$$
\end{lemma}

Since we will use the following observation on product measures of product actions repeatedly, we have included a proof for convenience of the reader.

\begin{lemma} \label{lemma.prod-measure}
Assume that $G_{i} \actson X_{i}$ and that $\Omega_{i}: G_{i}\times X_{i} \to \R$ is a continuous $1$-cocycle for $i=1,2$. Consider the product action $G_{1}\times G_{2} \actson X_{1}\times X_{2}$ and define the $1$-cocycle
$$
\Omega: (G_{1}\times G_{2}) \times ( X_{1}\times X_{2}) \to \R \;
$$
by $\Omega((g_{1}, g_{2}), (x_{1}, x_{2} ))=\Omega_{1}(g_{1}, x_{1})+\Omega_{2}(g_{2}, x_{2})$.

Assume that there exists a unique $e^{\beta \cdot \Omega_{1}}$-conformal probability measure $\nu$ on $X_{1}$. Then the map
$$
\mu\mapsto \nu\times \mu
$$
is an affine homeomorphism from the set of $e^{\beta \cdot \Omega_{2}}$-conformal measures $\mu$ on $X_{2}$ onto the set of $e^{\beta\cdot \Omega}$-conformal measures on $X_{1}\times X_{2}$.
\end{lemma}

\begin{proof}
Assume that $\mu$ is an $e^{\beta \cdot \Omega_{2}}$-conformal probability measure on $X_{2}$, and that $U_{i}\subset X_{i}$ is a Borel set for $i=1,2$. For $g_{i}\in G_{i}$ we have
$$
\nu\times \mu \left[ (g_{1}, g_{2}) \cdot (U_{1}\times U_{2}) \right]
= \nu(g_{1} \cdot U_{1}) \cdot \mu(g_{2}\cdot U_{2})
=\int_{U_{1}\times U_{2}} e^{\beta \cdot \Omega((g_{1}, g_{2}), (x_{1}, x_{2}))} \ \mathrm{d} (\nu\times\mu )\; ,
$$
from which it follows that the map is well defined. Since the map is clearly injective, affine and continuous, it suffices to argue that it is surjective.

Denote by $\pi_{i} : X_{1}\times X_{2} \to X_{i}$ the projection for $i=1,2$ and assume that $\eta$ is an $e^{\beta \cdot \Omega}$-conformal probability measure. Then, $(\pi_{2})_* \eta = \eta \circ \pi_{2}^{-1}$ is an $e^{\beta \cdot\Om_{2}}$-conformal probability measure on $X_{2}$, which we will denote by $\mu$. Let $F : X_{2} \recht [1,+\infty)$ be a continuous function and put $\rho = \int_{X_{2}} F \; d \mu$. The map
$$X_1 \supset \cU \mapsto \rho^{-1} \int_{\cU \times X_{2}} F\circ \pi_{2} \; d\eta$$
defines a probability measure $m$ on $X_{1}$, with the property that
$$
\int_{X_{1}} f \ \mathrm{d} m = \rho^{-1}\int_{X_{1}\times X_{2}} f\circ\pi_{1} \; F \circ \pi_{2} \ \mathrm{d} \eta \quad \text{ for all } f\in C(X_{1}) \ ,
$$
from which it follows that $m$ is $e^{\beta \cdot \Omega_{1}}$-conformal, i.e. $m=\nu$. This implies that
$$\int_{X_{1} \times X_{2}} H\circ \pi_{1} \; F \circ \pi_{2} \; d\eta = \int_{X_{1} \times X_{2}} H(x_{1}) \; F(x_{2}) \; d(\nu \times \mu)(x_{1},x_{2})$$
for all $H \in C(X_1)$ and $F \in C(X_{2})$ with $F \geq 1$. By linearity and density, we conclude that $\eta = \nu \times \mu$. This concludes the proof.
\end{proof}

\section{Groups of subexponential growth}

The purpose of this section is to prove Theorem \ref{theorem.A}. The strategy of proof is similar to the one used in the case $G=\Z$ presented in \cite{CT}, yet uses a different approach to construct conformal measures.

Let $G$ be a finitely generated group, let $S$ be a finite symmetric set that generates $G$ as a semi-group, and let $g\to |g|$ denote the corresponding word length function. Recall that $G$ is said to be of subexponential growth if
$$
\limsup_{k} \ \left \lvert \{ g\in G \ | \ |g|=k  \} \right \rvert^{1/k} =1 \; .
$$
In the following we will use the notation $\limsup_{g\to \infty} f(g)$ for any bounded function $f: G \to \R$ to denote the number
$$
\lim_{n\to \infty} \ \sup_{|g| >n} f(g) \; .
$$

\begin{lemma}\label{lem.subex}
Let $G$ be a finitely generated group of subexponential growth. Assume that $G\actson X$ and $\Omega : G \times X \to \mathbb{R}$ is a continuous $1$-cocycle. Let $\beta \in \R$. There exists a $\beta$-conformal measure if and only if there exists an $x\in X$ such that
\begin{equation} \label{eq.limsup}
\limsup_{g\to \infty}  \  \frac{\beta \cdot \Omega(g, x)}{|g|}  \leq 0 \ .
\end{equation}
\end{lemma}

\begin{proof}
To prove the first direction, let us assume that there exists a $\beta$-conformal measure $\mu$. By the Krein-Milman Theorem we can assume that $\mu$ is an extremal $\beta$-conformal measure, which in particular implies that it is an ergodic measure for the action $G \actson X$. Using the cocycle condition \eqref{eq.cocycle} it is straightforward to see that the real Borel function
$$
x \mapsto \limsup_{g\to \infty} \ \frac{\beta \cdot \Omega(g, x)}{|g|}
$$
is bounded and $G$-invariant, and hence it must be equal to a constant $c \in \R$ for $\mu$-a.e. $x\in X$. Assume for a contradiction that $c>0$. Setting
$$
N_{g} = \left\{ x\in X \ | \ \beta \cdot \Omega(g,x) \geq |g| \cdot c/2  \right\}
$$
for each $g\in G$, it then follows that for any $n\in \N$ we have
$$
X= \bigcup_{ |g|>n } N_{g}
$$
up to a $\mu$-null set. Using \eqref{eq.conformal-measure} we get for $g\in G$ that
$$
1 = \int_{X} e^{\beta \cdot\Omega(g, x)} \ \mathrm{d} \mu(x) \geq  \int_{N_{g}} e^{\beta \cdot \Omega(g, x)} \ \mathrm{d} \mu (x)
\geq \int_{N_{g}} e^{|g| \cdot c/2  } \ \mathrm{d} \mu (x)
=\mu(N_{g}) e^{|g| \cdot c/2  }
$$
from which it follows that $\mu(N_{g}) \leq e^{-|g| \cdot c/2}$. Set
$$
G_{k}=  \{g\in G \ | \ |g|= k \} \; .
$$
Since $G$ is of subexponential growth we can choose $M\in \N$ big enough that $|G_{k}| \leq e^{k\cdot c/4}$ for all $k\geq M$. For any $n\geq M$ this implies that
$$
1 =\mu(X) \leq \sum_{|g| > n} \mu(N_{g}) \leq \sum_{k > n} |G_{k}|  e^{- k \cdot c/2}
\leq \sum_{k > n}  e^{-k \cdot c/4} \to 0 \text{ for } n\to \infty
$$
which is a contradiction. Hence $c\leq 0$, which proves one implication in the Lemma.

Assume for a proof of the other implication that there exists a point $x\in X$ which satisfies \eqref{eq.limsup}. For any $s>0$ we can therefore pick a $N\in \N$ such that both $\beta \cdot\Omega(g,x) \leq |g|s/4$ when $|g| \geq N$ and $|G_{n}| \leq e^{n\cdot s/4}$ when $n\geq N$. For $n\geq N$ then
$$
\sum_{|g| = n} e^{\beta \cdot \Omega(g,x)-n\cdot s} \leq e^{-n \cdot s/2} \; ,
$$
and hence
$$
\sum_{g\in G} e^{\beta \cdot \Omega(g, x)-|g| \cdot s} = \sum_{n=0}^{\infty} \sum_{|g| =n} e^{\beta \cdot \Omega(g, x)-n \cdot s} < \infty \ .
$$
Let $\delta_{y}$ denote the Dirac measure concentrated at a point $y\in X$, and recall that $\phi_{g}$ denotes the homeomorphism on $X$ defined by an element $g\in G$. For each $s>0$ we can now define a finite measure $m_{s}$ on $X$ by
$$
m_{s} = \sum_{g\in G} e^{\beta \cdot \Omega(g, x)-|g| \cdot s} \ \delta_{\phi_{g}(x)} \ .
$$
For any $f\in C(X)$ and $h\in G$ we then get that
\begin{align*}
\int_X (f \circ \phi_{h}) \cdot e^{\beta \cdot \Omega(h, \cdot)} \ \mathrm{d} m_s
 &= \sum_{g \in G} f \circ \phi_{hg}(x) e^{\beta \cdot \Omega(h, \phi_{g}(x))+ \beta \cdot \Omega(g, x)} e^{-|g| \cdot s}  \\
& =  \sum_{g' \in G} f \circ \phi_{g'}(x) e^{\beta \cdot \Omega(g',x)} e^{-|h^{-1}g'|\cdot s} \; .
\end{align*}
Since $|h^{-1}g|\in [|g|-|h|, |g|+|h|]$ for any $g\in G$ we get that
\begin{align*}
\left| \int_X (f \circ \phi_{h}) \cdot e^{\beta \cdot \Omega(h, \cdot)} \ \mathrm{d} m_s  - \int_X f  \ \mathrm{d} m_s \right|
&\leq \sum_{g \in G} \lVert f \rVert_{\infty} e^{\beta \cdot \Omega(g,x)} \left| e^{-|h^{-1}g|\cdot s}-e^{-|g|\cdot s}\right| \\
&\leq\sum_{g \in G} \lVert f \rVert_{\infty} e^{\beta \cdot  \Omega(g,x)-|g| \cdot s} \left| e^{|h| \cdot s}-1\right| \\
&\leq \|f\|_\infty \, m_s(X) \, \left|e^{|h|\cdot s} - 1 \right| \; ,
\end{align*}
and it follows from this that any condensation point of the probability measures $m_{s}(X)^{-1} m_{s}$ for $s\to 0$ is a $\beta$-conformal measure. This proves the other direction.
\end{proof}

\begin{lemma}\label{lem.KMS-spectrum-pol-growth}
Let $G$ be a finitely generated group with subexponential growth and word length function $g \mapsto |g|$. Let $G \actson X$  and let $\Om : G \times X \recht \R$ be any continuous $1$-cocycle. Then the KMS spectrum $K$ must have one of the following forms: $\{0\}$, $[0,+\infty)$, $(-\infty,0]$ or $\R$, and can be determined as follows.
\begin{itemlist}
\item We have $[0,+\infty) \subset K$ iff there exists an $x \in X$ such that $\limsup_{g \to \infty} \Om(g,x)/|g| \leq 0$.
\item We have $(-\infty,0] \subset K$ iff there exists an $x \in X$ such that $\liminf_{g \to \infty} \Om(g,x)/|g| \geq 0$.
\end{itemlist}
\end{lemma}

\begin{proof}
If the condition in \eqref{eq.limsup} holds for a $\beta' \neq 0$ then it also holds for all $\beta\in \R$ with the same sign as $\beta '$. Since \eqref{eq.limsup} is always true for $\beta = 0$, the statement therefore follows from Lemma \ref{lem.subex} and Lemma \ref{lemma.states-conformal}.
\end{proof}

We can now prove Theorem \ref{theorem.A}.

\begin{proof}[Proof of Theorem \ref{theorem.A}]
Let $(G(n))_{n\in \N} \subset G$ be an increasing sequence of finitely generated subgroups with union $G$. Denote by $K_n$ the KMS spectrum of the action and $1$-cocycle restricted to $G(n)$. Denote by $K$ the KMS spectrum of $G \actson X$ and $\Om :G\times X \to \R$. By \eqref{eq.conformal-measure} then $K = \bigcap_n K_n$ and the conclusion follows from Lemma \ref{lem.KMS-spectrum-pol-growth}.
\end{proof}

It follows from Theorem 6.8 in \cite{CT} that for an action $\Z \actson X$ which is uniquely ergodic, the KMS spectrum is either $\{0\}$ or $\R$. In the following we will argue that the same is true when $G=\Z^{n}$ for some $n\in \N$.

Assume that $G \actson X$ is an action by homeomorphisms and that $\Om : G \times X \recht \R$ is a continuous $1$-cocycle. Whenever $\mu$ is a $G$-invariant probability measure on $X$, the formula
$$\Om_\mu : G \recht \R : \Om_\mu(g) = \int_X \Om(g,x) \, d\mu(x)$$
defines a group homomorphism from $G$ to $\R$.

\begin{proposition}
Let $G=\Z^n$ and let $G \actson X$ be an action by homeomorphisms of a compact space $X$. Let $\Om : G \times X \recht \R$ be any continuous $1$-cocycle.
\begin{enumlist}
\item If there exists an ergodic $G$-invariant probability measure $\mu$ on $X$ such that $\Om_\mu = 0$, then the KMS spectrum equals $\R$.
\item If $G \actson X$ is uniquely ergodic, with unique $G$-invariant probability measure $\mu$, then the KMS spectrum $K$ is either $\{0\}$ or $\R$. We have $K = \R$ iff $\Om_\mu = 0$, and we have $K = \{0\}$ iff $\Om_\mu \neq 0$.
\end{enumlist}
\end{proposition}

\begin{proof}
1.\ By \cite[Theorem 1]{BD89}, we have that
$$\lim_{g \recht \infty} \frac{\Om(g,x)}{|g|} = 0 \quad\text{for $\mu$-a.e.\ $x \in X$.}$$
So, we can pick $x \in X$ such that $\Om(g,x) / |g| \recht 0$. By Lemma \ref{lem.KMS-spectrum-pol-growth}, the KMS spectrum equals $\R$.

2.\ Denote by $K$ the KMS spectrum. If $\Om_\mu = 0$, it follows from the previous point that $K = \R$. Assume that $K \neq \{0\}$. We prove that $\Om_\mu = 0$. By Lemma \ref{lem.KMS-spectrum-pol-growth}, we find an $x \in X$ such that either $\limsup_{g \to \infty} \Om(g,x)/|g| \leq 0$ or $\liminf_{g \to \infty} \Om(g,x)/|g| \geq 0$. By symmetry, it suffices to deal with the first case.

Fix a nontrivial element $a \in G \setminus \{e\}$. We prove that $\Om_\mu(a) \leq 0$. Denote by $M$ the weak$^*$ closed convex hull of all probability measures $\nu$ on $X$ that can be obtained as a weak$^*$ limit point of a sequence of the form
$$\nu_m = \frac{1}{m} \sum_{k=0}^{m-1} \delta_{a^k g \cdot x}$$
where $g \in G$ is arbitrary.

Let $g \in G$ and let $\nu$ be any weak$^*$ limit point of the sequence $(\nu_m)_{m \in \N}$. Since
$$\Om(a^m g,x) = \Om(g,x) + \sum_{k=0}^{m-1} \Om(a,a^k g \cdot x)$$
and since there exists a $C > 0$ with $C^{-1} m - C \leq |a^m g| \leq C m + C$ for all $m \in \N$, it follows that
$$\limsup_{m \recht +\infty} \frac{1}{m} \sum_{k=0}^{m-1} \Om(a,a^k g \cdot x) \leq 0 \; .$$
We conclude that $\int_X \Om(a,y) \, d\nu(y) \leq 0$. Therefore $\int_X \Om(a,y) \, d\nu(y) \leq 0$ for all $\nu \in M$. By construction, $M$ is weak$^*$ closed, convex and globally $G$-invariant. So, $\mu \in M$ and we get that $\Om_\mu(a) \leq 0$.

Since $a \in G \setminus \{e\}$ was arbitrary and $\Om_\mu(a^{-1}) = - \Om_\mu(a)$, it follows that $\Om_\mu = 0$.
\end{proof}

\section{Group actions with prescribed KMS spectrum}

The goal of this section is to prove Theorem \ref{thm.all-closed-sets}.
We start from a group $\Lambda = \bigoplus_{n \in \N} \Lambda_n$ that is the direct sum of infinitely many nontrivial finite groups. We consider the minimal action of $\Lambda$ by translations on $X = \prod_{n\in \N} \Lambda_n$. There is a natural family of infinite product $1$-cocycles $\Lambda \times X \to \R$ such that for every $\beta \in \R$, there is a unique $\beta$-conformal probability measure $\mu_\beta$ on $X$. For any continuous function $H : X \to (0,+\infty)$, we consider the function
\begin{equation} \label{eq.kmsspectrumfunc}
\vphi : \beta \mapsto \int_{X} H^{\beta} \ \mathrm{d} \mu_{\beta} \; .
\end{equation}
We prove in Proposition \ref{prop.generic-construction} that the set $\{\beta \in \R \mid \vphi(\beta) = 1 \}$ arises as the KMS spectrum of the natural action of the wreath product $\Lambda \wr \Z$ on $X^\Z$. In Lemma \ref{lem.main-technical}, we prove that basically any continuous function $\vphi$ with $\vphi(0) = 1$ can be written as in \eqref{eq.kmsspectrumfunc}. Altogether, this then leads to the proof of the following more explicit version of the first part of Theorem \ref{thm.all-closed-sets}.

\begin{theorem}\label{thm.all-closed-sets-amenable}
Let $(\Lambda_n)_{n \in \N}$ be a sequence of nontrivial finite groups. Put $\Lambda = \bigoplus_{n \in \N} \Lambda_n$ and consider the product action of $\Lambda$ on $X = \prod_{n\in \N} \Lambda_n$ by translation.

The natural action of the wreath product $\Gamma = \Lambda \wr \Z = \Lambda^{(\Z)} \rtimes \Z$ on $Y = X^\Z$ is minimal, uniquely ergodic, topologically free and has the following universality property: for every closed subset $K \subset \R$ with $0 \in K$, there exists a continuous $1$-cocycle $\Gamma \times Y \recht \R$ whose KMS spectrum equals $K$, and which has a unique $\beta$-KMS state for every $\beta \in K$.
\end{theorem}

We then provide in Proposition \ref{prop.cantorset} a construction to realize sets of the form
$$\{\beta \in \R \mid \vphi_1(\beta) = 1/2 \;\;\text{and}\;\;\vphi_2(\beta) = 2 \} \; ,$$
with $\vphi_i$ as in \eqref{eq.kmsspectrumfunc}, as KMS spectrum of a natural action of a free product group. We prove in Lemma \ref{lemma.realizablefractions} that any closed subset $K \subset \R$ with $0 \not\in K$ can be obtained in this way. This then leads to the proof of the second part of Theorem \ref{thm.all-closed-sets} in Section \ref{sec.all-closed-sets-nonamenable}.

\subsection{Realizable functions}

\begin{definition}\label{def.realizable}
We say that a function $\varphi : \R \recht (0,+\infty)$ is \emph{realizable} by the action $\Lambda \actson X$ if there exists a continuous $1$-cocycle $\Om : \Lambda \times X \recht \R$ with the following properties. For every $\beta \in \R$, there is a unique $\be$-conformal probability measure $\mu_\be$ on $X$ and there exists a continuous function $H : X \recht (0,+\infty)$ such that
$$\vphi(\beta) = \int_X H(x)^\be \; d\mu_\be(x) \; .$$
\end{definition}

A key point is to prove that basically arbitrary continuous functions are realizable by actions of product groups. The following lemma is thus important for us.

\begin{lemma} \label{lem.main-technical}
Let $(\Lambda_n)_{n \in \N}$ be a sequence of nontrivial finite groups. Put $\Lambda = \bigoplus_{n\in \N} \Lambda_n$ and consider the product action of $\Lambda$ on $X = \prod_{n\in \N} \Lambda_n$ by translation. Whenever $a > 1$ and $\zeta \in C_0(\R,[-1/2,1/2])$, the function
$$\varphi(\beta) = 1 + \frac{a^\beta - 1}{a^\beta + 1} \, \zeta(\beta)$$
is realizable by $\Lambda \actson X$.
\end{lemma}

Before proving Lemma \ref{lem.main-technical}, we need some preparation.

\begin{lemma}\label{lem.cR}
Define the subset $\cR \subset C_0(\R,(0,+\infty))$ consisting of the functions of the form
\begin{equation} \label{eq.-r}
r : \R \recht (0,+\infty) : r(\beta) = \frac{\sum_{i=1}^N a_i^\beta}{ \sum_{j=1}^M b_j^\beta }
\end{equation}
where $N,M \in \N$, $a_i,b_j \in (0,+\infty)$ and $\max \{b_j\} > \max\{a_i\}$ and $\min \{b_j\} < \min\{a_i\}$ to ensure that $r$ tends to zero at infinity.

Then, $\cR $ is uniformly dense in $C_0(\R,[0,+\infty))$.
\end{lemma}

\begin{proof}
Denote by $\cRbar$ the uniform closure of $\cR$ inside $C_0(\R,[0,+\infty))$. We first prove that all functions
\begin{equation}\label{eq.more-r}
r : \R \recht (0,+\infty) : r(\beta) = \frac{\sum_{i=1}^N c_i a_i^\beta}{\sum_{j=1}^M d_j b_j^\beta}
\end{equation}
where $N,M \in \N$, $a_i,c_i,b_j,d_j \in (0,+\infty)$, $\max \{b_j\} > \max\{a_i\}$ and $\min \{b_j\} < \min\{a_i\}$, belong to $\cRbar$.

When the coefficients $c_i,d_j$ are strictly positive integers, the function $r$ simply belongs to $\cR $ because we may repeat terms in the numerator and denominator. Dividing numerator and denominator by the same strictly positive integer, we also find that $r \in \cR $ when the coefficients $c_i,d_j$ are strictly positive rational numbers. Approximating arbitrary coefficients $c_i,d_j$ by rational numbers, it follows that all the functions $r$ in \eqref{eq.more-r} belong to $\cRbar$. Note also that $\cRbar$ is closed under taking products and linear combinations with strictly positive coefficients.

To prove that $\cRbar = C_0(\R,[0,+\infty))$, by the Hahn-Banach theorem and the Riesz theorem, it is sufficient to prove the following statement: whenever $\mu$ is a finite signed measure on the Borel sets of $\R$ with
\begin{equation}\label{eq.cond1}
\int_\R r(x) \; d\mu(x) \geq 0 \quad\text{for all $r \in \cRbar$,}
\end{equation}
we must have that
\begin{equation}\label{eq.cond2}
\int_\R f(x) \; d\mu(x) \geq 0 \quad\text{for all $f \in C_c(\R,[0,+\infty))$.}
\end{equation}
Assume that $\mu$ is a finite signed measure satisfying \eqref{eq.cond1}. Define the approximate unit
$$\varphi_n : \R \recht (0,+\infty) : \varphi_n(x) = D_n^{-1} (2^x + 2^{-x})^{-n} \quad\text{with}\quad D_n = \int_\R (2^x + 2^{-x})^{-n} \; dx \; .$$
Note that $\varphi_n \in \cRbar$ and that, for all $y \in \R$, the function $x \mapsto \varphi_n(x-y)$ also belongs to $\cRbar$, since it is of the form \eqref{eq.more-r}.

Take an arbitrary $f \in C_c(\R,[0,+\infty))$. We have to prove that \eqref{eq.cond2} holds. Define $f_n = f * \varphi_n$. Since $f_n \recht f$ uniformly, it suffices to prove that all $f_n$ have positive integral w.r.t.\ $\mu$. But,
$$\int_\R f_n(x) \; d\mu(x) = \int_\R   f(y) \; \int_\R  \; \varphi_n(x-y) \; d\mu(x) \; dy \geq 0$$
because $x \mapsto \vphi_n(x-y)$ belongs to $\cRbar$ for all $y \in \R$. This ends the proof of the lemma.
\end{proof}

\begin{lemma}\label{lemma.help-technical}
Let $\{ j_{k}\}_{k \in \N} \subset \N$ be a sequence with $j_{k} \geq 2$ for all $k\in \N$ and let $t\in (0, \infty)$. Take $f\in C_{0}(\R, [-1/2,1/2])$ and $\varepsilon >0$. There exist functions $\eta_{1}, \eta_{2}\in C_{0}(\R, [0,1/2])$ with $\lVert f-(\eta_{1}-\eta_{2}) \rVert_{\infty} \leq \varepsilon$ and an $n\in \N$ such that the following hold.

There is set $F$ with $|F|=\prod_{k=1}^{n} j_{k}$, a probability measure $\mu$ of full support on $F$ and a partition $F = F_{0} \sqcup F_{1} \sqcup F_{2}$ such that
$$
\frac{1}{1+t^{\beta}}\cdot  \eta_{1}(\beta) = \frac{\sum_{h\in F_{0}}  \mu(h)^{\beta} }{ \sum_{g\in F} \mu(g)^{\beta}} \quad\text{and}\quad
\frac{t^{\beta}}{1+t^{\beta}} \cdot \eta_{2}(\beta) = \frac{\sum_{h\in F_{1}}  \mu(h)^{\beta} }{ \sum_{g\in F} \mu(g)^{\beta}}
$$
for all $\beta \in \R$.
\end{lemma}

\begin{proof}
Let us first notice that if a family of functions $\mathcal{F} \subset C_{0}(\R, [0,\infty) )$ is dense in $C_{0}(\R, [0,\infty))$ then the family $\{ g(1+2g)^{-1} : g\in \mathcal{F} \}$ is dense in $C_{0}(\mathbb{R}, [0, 1/2])$. This observation follows from the calculation
$$
\Bigl\| \frac{g_1}{1+2g_1} - \frac{g_2}{1+2g_2}\Bigr\|_\infty = \Bigl\| \frac{g_1 - g_2}{(1+2g_1)(1+2g_2)}\Bigr\|_\infty \leq \|g_1 - g_2 \|_\infty   \; .
$$
Now decompose $f=f_{+}-f_{-}$ in its positive and negative part. Since  $\cR$ is dense in $C_{0}(\R, [0, \infty))$ we can choose $\eta_{1}' \in C_{0}(\mathbb{R}, [0, 1/2])$ with $\lVert f_{+}-\eta_{1}' \rVert_{\infty} \leq \varepsilon /3$ such that $\eta_{1}'$ is on the form
$$
\eta_{1}'(\beta) = \frac{\sum_{i=1}^N a_i^\beta }{2\sum_{i=1}^N a_i^\beta  + \sum_{j=1}^M b_j^\beta}
$$
with $a_{i}, b_{j} \in (0,\infty)$. Since $\eta_{1}'\in C_{0}(\mathbb{R}, [0, 1/2])$ there exists an $R>0$ such that $|\eta_{1}'(\beta)| \leq \varepsilon/6$ for $|\beta| \geq R$. Since $t \in (0,\infty)$, we have $0< (1+t^{\beta})^{-1} <1$ for all $\beta \in \R$. Hence we can choose $L =  j_{1} \cdots j_{p} $ for some $p\in \N$ and $K \in \N$ sufficiently big that $L K^{-1} \geq 4N+2M$ yet $L K^{-1}$ is so close to $4N+2M$ that when $T:=L-4KN-2KM$ setting
$$
\eta_{1}(\beta) = \frac{K\sum_{i=1}^N a_i^\beta }{2K\sum_{i=1}^N a_i^\beta  + K\sum_{j=1}^M b_j^\beta + T (1+t^{\beta})^{-1}}
= \frac{\sum_{i=1}^N a_i^\beta }{2\sum_{i=1}^N a_i^\beta  + \sum_{j=1}^M b_j^\beta + K^{-1}T (1+t^{\beta})^{-1}}
$$
gives $|\eta_{1}(\beta)-\eta_{1}'(\beta)|\leq \varepsilon /6$ for $\beta \in [-R,R]$. By choice of $R$ this implies that $\lVert \eta_{1} -  \eta_{1}'\rVert_{\infty} \leq \varepsilon/6$, and hence $\lVert f_{+} - \eta_{1} \rVert_{\infty} \leq \varepsilon/2$.

In a similar way we can choose $\eta_{2}\in C_{0}(\mathbb{R}, [0, 1/2])$ such that
$$
\eta_{2}(\beta) = \frac{\sum_{i=1}^P c_i^\beta }{2\sum_{i=1}^P c_i^\beta  + \sum_{j=1}^Q d_j^\beta + B (1+t^{\beta})^{-1}}
$$
with $c_{i}, d_{j} \in (0,\infty)$, $\lVert f_{-}-\eta_{2} \rVert_{\infty} \leq \varepsilon /2$ and non-negative integers $P, Q, B$ with $4P+2Q+B = j_{p+1}\cdots j_{n}$ for some $n>p$. By re-choosing $N, M, P, Q \in \N$ and $a_{i} , b_{j}, c_{i}, d_{j} \in (0,\infty)$ it follows that we can write
$$
\frac{1}{1+t^{\beta}}\cdot\eta_{1}(\beta)
=\frac{\sum_{i=1}^N a_i^\beta }{2\sum_{i=1}^N a_i^\beta  + \sum_{j=1}^M b_j^\beta } \quad\text{and}\quad \frac{t^{\beta}}{1+t^{\beta}}\cdot\eta_{2}(\beta)
=\frac{\sum_{i=1}^P c_i^\beta }{2\sum_{i=1}^P c_i^\beta  + \sum_{j=1}^Q d_j^\beta }
$$
with $2N+M=j_{1} \cdots j_{p}$ and $2P+Q=j_{p+1} \cdots j_{n}$. Writing these two fractions with the same denominator, we see that this denominator contains four copies of $\sum_{i=1}^{N} \sum_{j=1}^{P} (a_{i} c_{j})^\beta$, while the two numerators contain two copies each. Since the other terms in the two nominators are different from each other, normalizing  this expression proves the Lemma.
\end{proof}

We are now ready to prove Lemma \ref{lem.main-technical}.

\begin{proof}[Proof of Lemma \ref{lem.main-technical}]
Assume that $f\in C_0(\R,[-1/2,1/2])$, $t\in (1, \infty)$ and $\varepsilon>0$. By applying Lemma \ref{lemma.help-technical} to the sequence $(|\Lambda_{n}|)_{n\in \N}$ we find functions $\eta_{1}, \eta_{2} \in C_0(\R,[0,1/2])$ such that $\lVert f - (\eta_{1}-\eta_{2}) \rVert_{\infty} \leq \varepsilon$ and a probability measure $\mu$ of full support on $F:=\prod_{i=1}^{n} \Lambda_{i}$ for some $n \in \N$ such that
$$
\frac{1}{1+t^{\beta}}\cdot  \eta_{1}(\beta) = \frac{\sum_{h\in F_{0}}  \mu(h)^{\beta} }{ \sum_{g\in F} \mu(g)^{\beta}} \quad\text{and}\quad
\frac{t^{\beta}}{1+t^{\beta}} \cdot \eta_{2}(\beta) = \frac{\sum_{h\in F_{1}}  \mu(h)^{\beta} }{ \sum_{g\in F} \mu(g)^{\beta}}
$$
for all $\beta \in \R$ and some partition $F_{0} \sqcup F_{1} \sqcup F_{2}$ of $F$. Define a $1$-cocycle
$$
\Omega: F \times F \to \R \; : \; \Omega(g,h)=\log(\mu(gh))-\log(\mu(h)) \; .
$$
For each $\beta \in \R$, Lemma \ref{lemma.cohomologous} implies that the unique $\beta$-conformal measure $\mu_{\beta}$ for this cocycle then satisfies
$$
\mu_{\beta} (h) = \frac{ \mu(h)^{\beta} }{ \sum_{g\in F} \mu(g)^{\beta}} \quad \text{ for all } h\in F \; .
$$
Defining a function $H: \prod_{i=1}^{n} \Lambda_{i} \to [t^{-1}, t]$ by
$$
H(x)=
\begin{cases}
t \ & \ \text{if } x\in F_{0} \; , \\
t^{-1} \ & \ \text{if } x\in F_{1} \; , \\
1  \ & \ \text{if } x\in F_{2} \; ,
\end{cases}
$$
it follows from Lemma \ref{lemma.help-technical} and a direct computation that for this choice of cocycle $\Omega$ and function $H$ we have
$$
1+\frac{t^{\beta}-1}{t^{\beta}+1}(\eta_{1}(\beta)- \eta_{2}(\beta)) = \int_{\prod_{i=1}^{n} \Lambda_{i}} H^{\beta} \ \mathrm{d} \mu_{\beta} \text{ for all $\beta \in \R$}.
$$
Hence the function $\beta \mapsto 1+\frac{t^{\beta}-1}{t^{\beta}+1}(\eta_{1}(\beta)- \eta_{2}(\beta))$ is realizable by the action $\bigoplus_{i=1}^{n} \Lambda_{i} \actson \prod_{i=1}^{n} \Lambda_{i}$.

We will now use this observation to prove the Lemma. Let $(a_{n})_{n\in \N} \subset (1, \infty)$ be a sequence with $a_{1}=a$ such that $\sum_{n=1}^{\infty}(a_{n}-1) < \infty$, and define for all $n\in \N$ the function
$$P_n : \R \recht \R : P_n(\beta) = \frac{a_n^\beta - 1}{a_n^\beta + 1} \; .$$
Note that $\|P_n\|_\infty \leq 1$. Also note that $P_n(\beta) / P_{n+1}(\beta)$ can be continuously extended to $\beta = 0$ and tends to $1$ when $\beta \recht \pm\infty$. We define $C_n = \|P_n / P_{n+1}\|_\infty < \infty$.

Put $\psi_0 = 1$ and $l_{0}=1$. We now prove that we can inductively choose for all $k \in \N$ a number $l_{k}>l_{k-1}$ and an element $\zeta_k \in C_0(\R,[-1/2,1/2])$ such that, writing
$$\psi_k = \psi_{k-1} (1 + P_k \zeta_k) \; ,$$
we have that
\begin{align}
&\|\psi_k\|_\infty \leq 2 \quad , \label{cond.1}\\
&\Bigl\| \frac{\varphi - \psi_{k-1}}{\psi_{k-1} P_k} - \zeta_k \Bigr\|_\infty \leq \min \{ (4 C_k)^{-1} , 2^{-k} \} \; ,\label{cond.2}
\end{align}
and that the function $1+P_{k}\zeta_{k}$ is realizable by the action $\bigoplus_{i=l_{k-1}}^{l_{k}-1} \Lambda_{i} \actson \prod_{i=l_{k-1}}^{l_{k}-1} \Lambda_{i}$ and a function $H_{k}:\prod_{i=l_{k-1}}^{l_{k}-1} \Lambda_{i} \to [a_{k}^{-1}, a_{k}]$.

To choose $\zeta_{1}\in C_0(\R,[-1/2,1/2])$ we apply the first part of the proof with $f=\zeta$, $t=a_{1}$ and $\varepsilon = \min\{ (4C_{1})^{-1}, 2^{-1}\}$ and we set $\zeta_{1} = \eta_{1} - \eta_{2}$ and $l_{1}=n+1$ where $1+P_{1}\zeta_{1}$ is realizable by $\bigoplus_{i=1}^{n} \Lambda_{i} \actson \prod_{i=1}^{n} \Lambda_{i}$. By the choice of $\zeta_{1}$ and $P_{1}$ we have that $1/2 \leq \psi_{1} \leq 3/2$, so since
$$
\frac{\varphi-\psi_{0}}{P_{1} \psi_{0}} = \zeta
$$
it follows that the statement in our inductive process is true for $k=1$.

Assume that we have chosen $\zeta_1,\ldots,\zeta_n $ and $l_{1} < l_{2} < \cdots <l_{n}$ satisfying our conditions for all $k \in \{1,\ldots,n\}$. Since
$$\frac{\vphi - \psi_n}{\psi_n P_{n+1}} = \frac{\vphi - \psi_{n-1} - \psi_{n-1} P_n \zeta_n}{\psi_n P_{n+1}} = \frac{1}{1 + P_n \zeta_n} \, \frac{P_n}{P_{n+1}} \, \Bigl( \frac{\vphi - \psi_{n-1}}{\psi_{n-1} P_n} - \zeta_n\Bigr)$$
and since $1 + P_n \zeta_n \geq 1/2$, we get that
$$\Bigl\|\frac{\vphi - \psi_n}{\psi_n P_{n+1}}\Bigr\|_\infty \leq 2 \, C_n \, \frac{1}{4 C_n} = \frac{1}{2} \; .$$
Since all $\zeta_1,\ldots,\zeta_n$ tend to zero at infinity and since all $P_k$ are bounded, we get that $\psi_n$ tends to $1$ at infinity. Also $\varphi$ tends to $1$ at infinity. Hence, $(\vphi - \psi_n)/(\psi_n P_{n+1})$ tends to zero at infinity. We now apply the observation in the first part of the proof again with $f=(\vphi - \psi_n)/(\psi_n P_{n+1})$, $t=a_{n+1}$, $\varepsilon = \min\{(4C_{n+1})^{-1}, 2^{-n-1}\}$ and the sequence $(\Lambda_{j})_{j\geq l_{n}}$ to obtain an element $\zeta_{n+1} \in C_0(\R,[-1/2,1/2])$ and a number $l_{n+1} > l_{n}$ such that \eqref{cond.2} is true for $k=n+1$ and $1+P_{n+1}\zeta_{n+1}$ is realizable by the action $\bigoplus_{i=l_{n}}^{l_{n+1}-1} \Lambda_{i} \actson \prod_{i=l_{n}}^{l_{n+1}-1} \Lambda_{i}$ and a function $H_{n+1}:\prod_{i=l_{n}}^{l_{n+1}-1} \Lambda_{i} \to [a_{n+1}^{-1}, a_{n+1}]$. We then define $\psi_{n+1} = \psi_n (1 + P_{n+1} \zeta_{n+1})$. Since \eqref{cond.2} holds for $k=n+1$, we get that
$$\Bigl\| \frac{\varphi - \psi_{n+1}}{\psi_n P_{n+1}} \Bigr\|_\infty \leq 2^{-n-1} \; . $$
It follows that
\begin{equation}\label{eq.good-approx}
\|\varphi - \psi_{n+1} \|_\infty \leq \|\psi_n\|_\infty \, \|P_{n+1}\|_\infty \, 2^{-n-1} \leq 2^{-n} \; .
\end{equation}
Since $\|\varphi\|_\infty \leq 3/2$ and $n \geq 1$, we get that $\|\psi_{n+1}\|_\infty \leq 2$, so that also \eqref{cond.1} holds for $k=n+1$. This concludes the inductive construction of the functions $\zeta_n $, the numbers $l_{n}$ and the corresponding functions $\psi_n$.

By definition we have that
\begin{equation}\label{eq.product}
\psi_n = \prod_{j=1}^{n} (1+P_{j} \zeta_{j}) \; ,
\end{equation}
and by \eqref{eq.good-approx} it follows that $\psi_{n}$ converges to $\varphi$. Setting $\Lambda_{n}'= \prod_{i=l_{n-1}}^{l_{n}-1} \Lambda_{i}$ we have that $\Lambda = \bigoplus_{n\in\N} \Lambda_{n}'$ and $X=\prod_{n\in \N} \Lambda_{n}'$. For each $n\in \N$ let $\Omega_{n}$ be the $1$-cocycle on $\Lambda_{n}' \times \Lambda_{n}'$ that has as unique $\beta$-conformal probability measure $\mu_{n, \beta}$ for all $\beta \in \R$ such that
$$
1+P_{n}(\beta) \zeta_{n}(\beta) = \int_{\Lambda_{n}'} H_{n}^{\beta} \ \mathrm{d} \mu_{n, \beta} \quad \text{ for all $\beta \in \R$.}
$$
Since $\sum_{n=1}^{\infty}(a_{n}-1) < \infty$ we can define a continuous function $H$ on $X$ by setting
$$
H(y) = \prod_{n=1}^{\infty} H_{n}(y_{n}) \quad \text{ for } y \in X=\prod_{n=1}^{\infty} \Lambda_{n}'\; ,
$$
and we can define a cocycle $\Omega: \Lambda \times X \to \R$ by setting
$$
\Omega((\lambda_{n})_{n\in \N}, (x_{n})_{n\in \N}) = \sum_{n\in \N} \Omega_{n}(\lambda_{n}, x_{n}) \; .
$$
It follows from repeated use of Lemma \ref{lemma.prod-measure} that $\mu_{\beta}:=\prod_{n\in \N} \mu_{n, \beta}$ is the unique $\beta$-conformal measure on $X$, and hence a direct computation using \eqref{eq.product} shows that
$$
\varphi(\beta) = \int_{X} H^{\beta} \ \mathrm{d} \mu_{\beta} \quad \text{ for all $\beta \in \R$}.
$$
This concludes the proof.
\end{proof}

In the proof of the second part of Theorem \ref{thm.all-closed-sets}, we will need the following complement of Lemma \ref{lem.main-technical}, providing more realizable functions.

\begin{lemma} \label{lemma.realizablefractions}
Let $(\Lambda_n)_{n \in \N}$ be a sequence of nontrivial finite groups. Put $\Lambda = \bigoplus_{n\in \N} \Lambda_n$ and consider the product action of $\Lambda$ on $X = \prod_{n \in \N} \Lambda_n$ by translation.

If $K \subset \R$ is a closed set that does not contain $0$ and $k\in \N$ with $k\geq 2$, there exist realizable functions $\vphi_1$ and $\vphi_2$ for the action $\Lambda \actson X$ such that
$$\{\beta \in \R \mid  \vphi_1(\beta) = k^{-1}\} = K = \{\beta \in \R \mid \vphi_2(\beta) = k \} \; .$$
\end{lemma}

\begin{proof}
Take $\delta > 0$ such that $K \cap [-\delta,\delta] = \emptyset$. Choose a continuous function $\psi : \R \recht [0,1]$ such that $\psi(\beta) = 1$ if and only if $\beta \in K$. Define the piecewise linear function
$$F : \R \recht \R : F(t) = \begin{cases} t &\;\;\text{if $|t| \leq 1/2$,}\\
1-t &\;\;\text{if $1/2 \leq t \leq 1$,}\\
-1-t &\;\;\text{if $-1 \leq t \leq -1/2$,}\\
0 &\;\;\text{if $|t| \geq 1$.}\end{cases}$$

We can assume without loss of generality that $|\Lambda_{1}|\geq 2k$, and we set $l=|\Lambda_{1}|-2k$. To construct $\vphi_1$, choose $b > 1$ large enough such that $2 (k-1)b^\beta +(l-k^{-1}l)b^{\beta} \leq 1/4$ for all $\beta \leq -\delta$. Set $c=b+1$, and take $a > c$ large enough such that
$$2(k-1)(b/a)^\beta + (l-k^{-1}l)(c/a)^{\beta}\leq \frac{1}{4} \;\;\text{for all $\beta \geq \delta$, and}\quad \frac{a^\delta + 1}{a^\delta -1} \leq 2 \; .$$
It follows that
$$\Bigl|\frac{a^\beta + 1}{a^\beta -1}\Bigr| \leq 2 \quad\text{if $|\beta| \geq \delta$.}$$
Define for all $\beta \neq 0$,
$$Q(\beta) = - \frac{2(k-1) b^\beta+(l-k^{-1}l)c^{\beta}}{1 + 2(k-1) b^\beta + a^\beta + l c^{\beta}} \, \cdot \frac{a^\beta + 1}{a^\beta - 1} \; .$$
Note that $|Q(\beta)| \recht +\infty$ if $|\beta| \recht 0$. We have made our choices such that $|Q(\beta)| \leq 1/2$ if $|\beta| \geq \delta$.

Then define $\zeta(\beta) = F(Q(\beta)) \, \psi(\beta)$ and put $\zeta(0) = 0$. Note that $\zeta$ is continuous. Since $Q(\beta)\to 0$ for $\beta \to \pm \infty$ we have $\zeta \in C_0(\R,[-1/2,1/2])$. Define
$$P(\beta) = \frac{a^\beta - 1}{a^\beta + 1} \; .$$
By Lemma \ref{lem.main-technical} the function $1 + P(\beta) \zeta(\beta)$ is realizable for the action $\bigoplus_{n\geq 2}\Lambda_{n} \actson \prod_{n\geq 2} \Lambda_{n}$. Since $2k+l=|\Lambda_{1}|$ it is straightforward to check that the function
$$\vphi_{1}: \R \to \R \; : \; \vphi_1(\beta) = \frac{1 + 2(k-1) b^\beta + a^\beta+lc^{\beta}}{k  + k a^\beta +lc^{\beta}} \, (1 + P(\beta) \zeta(\beta))$$
is realizable for $\Lambda \actson X$.

By construction, $\vphi_1(\beta) = k^{-1}$ if and only if
$$P(\beta) \zeta(\beta) = - \frac{2(k-1) b^\beta+(l-k^{-1}l)c^{\beta}}{1+2(k-1) b^\beta + a^\beta+lc^{\beta}} \; .$$
This holds iff $\beta \neq 0$ and $\zeta(\beta) = Q(\beta)$. Since $|F(Q(\beta))| < |Q(\beta)|$ when $|Q(\beta)| > 1/2$ and since $\psi(\beta) < 1$ when $\beta \not\in K$, we find that $\zeta(\beta) = Q(\beta)$ if and only if $|Q(\beta)| \leq 1/2$ and $\beta \in K$. But, for $\beta \in K$, we have $|\beta| \geq \delta$ and thus $|Q(\beta)| \leq 1/2$. We have proven that $\vphi_1(\beta) = k^{-1}$ if and only if $\beta \in K$.

To construct $\vphi_2$, choose $a>c>b > 1$ as before. Define for all $\beta \neq 0$,
$$
Q(\beta) = \frac{2 (k-1)b^\beta+(l-k^{-1}l)c^{\beta}}{1+a^{\beta}+k^{-1}lc^{\beta} } \; \cdot \frac{a^{\beta}+1}{a^\beta - 1} \; .
$$
It follows that $|Q(\beta)| \leq 1/2$ whenever $|\beta| \geq \delta$. Put $\zeta(\beta) = F(Q(\beta)) \, \psi(\beta)$ and put $\zeta(0) = 0$. By construction, $\zeta \in C_0(\R,[-1/2,1/2])$. Define
$$P(\beta) = \frac{a^\beta - 1}{a^\beta + 1} \; .$$
By Lemma \ref{lem.main-technical} the function $1 + P(\beta) \zeta(\beta)$ is realizable for the action $\bigoplus_{n\geq 2}\Lambda_{n} \actson \prod_{n\geq 2} \Lambda_{n}$, from which it follows that the product function
$$\vphi_2(\beta) = \frac{k + ka^\beta + lc^{\beta} }{1 + 2(k-1) b^\beta + a^\beta + lc^{\beta}} \, (1 + P(\beta) \zeta(\beta))$$
is realizable for $\Lambda \actson X$. By construction, $\vphi_2(\beta) = k$ if and only if
$$P(\beta) \zeta(\beta) = \frac{2 (k-1)b^\beta+(l-k^{-1}l)c^{\beta}}{1+a^{\beta}+k^{-1}lc^{\beta} } \; .$$
This holds iff $\beta \neq 0$ and $\zeta(\beta) = Q(\beta)$. As above, this holds if and only if $\beta \in K$.
\end{proof}

\subsection{KMS spectra for actions of amenable groups}

We consider the wreath product $\Lambda \wr \Z$ of a countable group $\Lambda$ with the group $\Z$ of integers. Recall that $\Lambda \wr \Z$ is defined as the semidirect product $\Lambda^{(\Z)} \rtimes \Z$, where $\Z$ acts on $\Lambda^{(\Z)}=\bigoplus_{n\in \Z} \Lambda$ by the shift. We denote by $\pi_n : \Lambda \recht \Lambda^{(\Z)}$ the embedding in the direct summand corresponding to $n \in \Z$. We view $\Z$ as a subgroup of $\Lambda \wr \Z$. We then have $k \pi_n(\lambda) = \pi_{n+k}(\lambda) k$ for all $k,n \in \Z$ and $\lambda \in \Lambda$.

\begin{proposition}\label{prop.generic-construction}
Let $\Lambda$ be a nontrivial countable group and $\Lambda \actson X$ a free action of $\Lambda$ by homeomorphisms of a compact metric space $X$. Consider the wreath product group $\Gamma = \Lambda \wr \Z$ and its natural action $\Gamma \actson Y = X^\Z$.

The action $\Gamma \actson Y$ has the following universality property: whenever the function $\vphi : \R \to (0,+\infty)$ is realizable by $\Lambda \actson X$, there exists a continuous $1$-cocycle $\Gamma \times Y \recht \R$ with KMS spectrum $K = \{\beta \in \R \mid \vphi(\beta) = 1 \}$ and with for every $\beta \in K$, a unique $\beta$-conformal probability measure $\mu_\beta$. Also, $\Gamma \actson (Y,\mu_\beta)$ is essentially free for every $\beta \in K$.
\end{proposition}

\begin{proof}
Let $\vphi : \R \to (0,+\infty)$ be a function that is realizable by $\Lambda \actson X$. By definition, there is a continuous $1$-cocycle $\Omega_0 : \Lambda \times X \recht \R$ and a continuous function $H : X \to (0,+\infty)$ such that for every $\beta \in \R$, there is a unique $\beta$-conformal measure $\nu_\beta$ on $X$ satisfying
$$\varphi(\beta) = \int_X H^\beta \, d\nu_\be \; .$$
Define the $1$-cocycle
$$\Om_1 : \Lambda \times X \recht \R : \Om_1(\lambda,x) = \log(H(\lambda \cdot x)) - \log(H(x)) + \Omega_0(\lambda,x) \; ,$$
which is cohomologous to $\Om_{0}$. Define $\Omega$ on $\Lambda^{(\Z)}\times Y$ by
$$\Omega(\lambda, x)=\sum_{n\leq 0} \Om_{1}(\lambda_{n}, x_{n})+\sum_{n>0} \Om_{0}(\lambda_{n}, x_{n})$$ and define $\Omega$ on $\Z\times Y$ by
$$
\Omega(n,x) =
\begin{cases}
\displaystyle\sum_{i=0}^{n-1} -\log(H(x_{-i})) \ & \ \text{ if } n> 0 \; , \\
0 \ & \ \text{ if } n= 0 \; , \\
\displaystyle\sum_{i=1}^{-n} \log(H(x_{i})) \ & \ \text{ if } n < 0 \; .
\end{cases}
$$
It is then straightforward to check that
$$
\Omega((\lambda, n) , x):= \Omega(\lambda, n\cdot x)+\Omega(n, x) \quad \text{for }(\lambda, n) \in \Lambda \wr \Z \text{ and }x\in Y
$$
is a continuous $1$-cocycle for the action $\Gamma \actson Y$.

For every $\beta \in \R$, it follows from Lemma \ref{lemma.cohomologous} that there is a unique $e^{\beta \cdot \Om_{1}}$-conformal probability measure $\eta_\beta$ on $X$ and it is given by
$$\frac{d\eta_\beta}{d\nu_\beta} = \varphi(\beta)^{-1} H^\beta \; .$$
By repeated use of Lemma \ref{lemma.prod-measure} it follows that for all $\beta \in \R$, the product measure
$$\mu_\beta = \prod_{n \leq 0} \eta_\beta \times \prod_{n > 0} \nu_\beta$$
is the unique $\beta$-conformal probability measure for the action $\Lambda^{(\Z)} \actson Y$ and the restriction of $\Om$ to $\Lambda^{(\Z)} \times Y$. For all $\beta \in \R$, the shift $\Z \actson (Y,\mu_\beta)$ is nonsingular and satisfies
$$\frac{d((-1) \cdot \mu_\beta)}{d\mu_\beta}(x) = \varphi(\beta) H(x_0)^{-\beta} \; .$$
It follows from Lemma \ref{lem.genconformal} that $\mu_\beta$ is $\beta$-conformal for the entire action $\Gamma \actson Y$ and $\Om : \Gamma \times Y \recht \R$ if and only if $\varphi(\beta) = 1$.

To prove that the action $\Gamma \actson (Y,\mu_\beta)$ is essentially free, it suffices to prove that for any $(\lambda, n) \in \Gamma\setminus \{e\}$ the set of $(\lambda, n)$ fixed points is a $\mu_{\beta}$-null set. Since $\Lambda$ acts freely on $X$, the set of fixed points is empty unless $n\neq 0$. When $n\neq 0$ there exists an $N>0$ such that any fixed point $(x_{k})_{k\in \Z}$ of $(\lambda, n)$ satisfies $x_{k}=x_{k+n}$ for all $k>N$. Since $\Lambda$ is nontrivial, $\nu_{\beta}$ is not a Dirac-measure, and it follows that the set of such eventually periodic points is a $\mu_{\beta}$-null set.
\end{proof}

Combining Proposition \ref{prop.generic-construction} with Lemma \ref{lem.main-technical}, it is now straightforward to prove Theorem \ref{thm.all-closed-sets-amenable}.

\begin{proof}[{Proof of Theorem \ref{thm.all-closed-sets-amenable}}]
Let $K\subset \R$ be a closed set with $0\in K$. Fix a $t\in (1, \infty)$ and define the function
$$
\varphi: \R \to \R \; : \; \varphi(\beta)=1+\frac{t^{\beta}-1}{t^{\beta}+1} \frac{d(\beta, K)}{2(1+\beta^{2})} \; .
$$
Since $0\in K$, we have $d(\beta, K)\leq |\beta|$. By Lemma \ref{lem.main-technical}, the function $\varphi$ is realizable for the action $\Lambda \actson X$. By construction, $\varphi(\beta)=1$ if and only if $\beta \in K$. Proposition \ref{prop.generic-construction} then provides a continuous $1$-cocycle $\Gamma \times X^{\Z} \to \R$ with KMS spectrum $K$ and a unique $\beta$-conformal measure $\mu_{\beta}$ for each $\beta \in K$, which is also essentially free. By construction, the action $\Gamma \actson X^{\Z}$ is minimal, uniquely ergodic and topologically free. Using Lemma \ref{lemma.cstaralgeba}, the theorem follows.
\end{proof}

The group $\Lambda$ and the space $X$ appearing in Theorem \ref{thm.all-closed-sets-amenable} give rise to a crossed product that can be interpreted as a UHF algebra. We can thus reformulate Theorem \ref{thm.all-closed-sets-amenable} in C$^*$-algebraic language as a result on KMS states on the crossed product C$^*$-algebra of the Bernoulli shift on the infinite tensor product of an UHF algebra.

\begin{theorem} \label{thmuhfalg}
Let $\mathcal{A}$ be an infinite dimensional UHF algebra. Set $\mathcal{B} = \bigotimes_{n\in \Z} \mathcal{A}$ and consider the Bernoulli action of $\Z$ on $\mathcal{B}$.
For any closed set $K \subset \R$ with $0 \in K$ there exists a continuous $1$-parameter group of automorphisms on $\mathcal{B} \rtimes \Z$ such that the KMS spectrum equals $K$.
\end{theorem}

\begin{proof}
Let $\{k_{n}\}_{n=1}^{\infty}$ be a strictly increasing sequence of integers such that $k_{n}$ divides $k_{n+1}$ for all $n \in \N$ and such that
$$
\mathcal{A} = \overline{  \bigcup_{n=1}^{\infty}M_{k_{n}}(\C) } \: .
$$
Set $k_{0}=1$ and let $\Lambda_{n}$ be a finite group with $|\Lambda_{n}|=k_{n}/k_{n-1}$ for all $n\in \N$. Then $\mathcal{A} \simeq C(X) \rtimes_{r} \Lambda$ where $X$ and $\Lambda$ are defined as in Theorem \ref{thm.all-closed-sets-amenable}. We thus have a canonical isomorphism
$$
C(X^{\Z}) \rtimes_{r} \Lambda^{(\Z)} \simeq \bigotimes_{n\in \Z} \Big( C(X) \rtimes_{r} \Lambda \Big)  \simeq \bigotimes_{n\in \Z}  \mathcal{A} = \mathcal{B} \; .
$$
The Bernoulli shift on $\mathcal{B}$ corresponds to an action $\alpha$ of $\mathbb{Z}$ on $C(X^{\Z}) \rtimes_{r} \Lambda^{(\Z)}$ such that
$$
\Big [C(X^{\Z}) \rtimes_{r} \Lambda^{(\Z)} \Big] \rtimes_{\alpha} \mathbb{Z} \simeq C(X^{\Z}) \rtimes_{r} (\Lambda^{(\Z)} \rtimes \Z ) \; .
$$
Applying Theorem \ref{thm.all-closed-sets-amenable}, the theorem is proven.
\end{proof}

\subsection{KMS spectra for actions of non-amenable groups}\label{sec.all-closed-sets-nonamenable}

To prove the second part of Theorem \ref{thm.all-closed-sets}, we start with the following counterpart of Proposition \ref{prop.generic-construction}.
The action that we construct is minimal, but not topologically free. To amend this we will add an extra space to the action. This is done in Lemma \ref{lemma.dummyaction} below.

\begin{proposition} \label{prop.cantorset}
Let $\Lambda_0$ be a nontrivial finite group and, for $i=1,2$, let $\Lambda_i \actson X_i$ be minimal actions of countable groups by homeomorphisms of a compact metric space. Put $\Lambda = \Lambda_0 \times \Lambda_1 \times \Lambda_2$ and $X = \Lambda_0 \times X_1 \times X_2$. Consider the product action of $\Lambda^{(\Z)}$ on $X^\Z$.

There exists a homeomorphism of $X^{\Z}$ such that the resulting minimal action of $G = \Lambda^{(\Z)} \ast \Z$ on $X^{\Z}$ has the following universality property: whenever for $i=1,2$, the functions $\vphi_i : \R \to (0,+\infty)$ are realizable by the actions $\Lambda_i \actson X_i$, there exists a continuous $1$-cocycle $G \times X^{\Z} \recht \R$ with KMS spectrum
$$K = \bigl\{\beta \in \R \bigm| \vphi_1(\beta) = |\Lambda_0|^{-1} \;\;\text{and}\;\; \vphi_2(\beta) = |\Lambda_0| \bigr\}$$
and with for every $\beta \in K$, a unique $\beta$-conformal probability measure $\mu_\beta$.
\end{proposition}

\begin{proof}
Partition $\Lambda_0^{\Z}$ into the clopen sets
$$Y_0 = \{x \in \Lambda_0^\Z \mid x_0 = x_1 = e\} \quad , \quad Y_1 = \{x \in \Lambda_0^\Z \mid x_0 \neq e\} \quad ,
\quad Y_2 = \{x \in \Lambda_0^\Z \mid x_0 = e , x_1 \neq e \} \; .$$
Define the homeomorphism $\theta_0 : Y_1 \to Y_2$ by
$$(\theta_0(x))_n = \begin{cases} x_n &\;\;\text{if $n < 0$,}\\ e &\;\;\text{if $n=0$,}\\ x_{n-1} &\;\;\text{if
$n > 0$.}\end{cases}$$
For $i=1,2$, let $\psi_i$ be the shift map on $X_i^{\Z}$. View $X^\Z = \Lambda_0^\Z \times X_1^\Z \times X_2^\Z$. Define the homeomorphism $\theta: X^{\Z}\to X^{\Z}$ by
$$
\theta(x,y,z) =
\begin{cases}
(x,y,z) &\;\;\text{if $x \in Y_0$,} \\
(\theta_0(x),y,\psi_2(z)) &\;\;\text{if $x \in Y_1$,}\\
(\theta_0^{-1}(x),\psi_1(y),z) &\;\;\text{if $x \in Y_2$.}
\end{cases}
$$
We combine the product action $\Lambda^{(\Z)} \actson X^\Z$ and the homeomorphism $\theta$ into a minimal action of $\Lambda^{(\Z)} \ast \Z \actson X^{\Z}$.

By definition, for $i=1,2$, we have continuous $1$-cocycles $\Om'_i : \Lambda_i \times X_i \to \R$ and continuous functions $H_i : X_i \to (0,+\infty)$ such that the following holds. For every $\be \in \R$, there is a unique $\be$-conformal probability measure $\nu_{i,\be}$ on $X_i$ and
$$\vphi_i(\beta) = \int_{X_i} H_i^\be \; d\nu_{i,\beta} \; .$$
Define the $1$-cocycles $\Om_i : \Lambda_i^{(\Z)} \times X_i^{\Z} \to \R$ by
$$
\Om_i(\pi_{n}(\lambda), x) =
\begin{cases}
\Omega_{i}'(\lambda, x_{n}) \ & \ \text{ if } n\geq 0 \; , \\
\Omega_{i}'(\lambda, x_{n})+\log(H_{i}(\lambda \cdot x_{n}))-\log(H_{i}(x_{n}))  \ & \ \text{ if } n< 0 \; ,
\end{cases}
$$
for $\lambda \in \Lambda_{i}$, where $\pi_n : \Lambda_i \to \Lambda_i^{(\Z)}$ denotes the embedding as $n$-th direct summand. For every $\beta \in \R$, there is a unique $\beta$-conformal measure $\mu_{i, \beta}$ on $X_{i}^{\Z}$, and this measure satisfies
$$
\frac{d((\psi_{i})_{*} \mu_{i, \beta})}{d\mu_{i, \beta}}(x) = \vphi_i(\beta)^{-1} \, H_i(x_0)^\beta \quad\text{for all $x \in X_i^\Z$.}$$

Let $a\in \Lambda^{(\Z)} \ast \Z$ denote the element $-1\in \Z$. Viewing $\Lambda^{(\Z)}$ as the direct product $\Lambda_0^{(\Z)} \times \Lambda_1^{(\Z)} \times \Lambda_2^{(\Z)}$ and viewing $X^\Z$ as the product $\Lambda_0^\Z \times X_1^\Z \times X_2^\Z$, uniquely define the continuous $1$-cocycle $\Om : (\Lambda^{(\Z)} \ast \Z) \times X^{\Z} \recht \R$ by
$$\Om((g,h,k),(x,y,z)) = \Om_1(h,y) + \Om_2(k,z) \quad\text{and}\quad \Om(a,(x,y,z)) = \begin{cases} 0 &\;\;\text{if $x \in Y_0$,}\\
\log(H_1(y_0)) &\;\;\text{if $x \in Y_1$,}\\
\log(H_2(z_0)) &\;\;\text{if $x \in Y_2$.}\end{cases}$$
Denote by $\nu_0$ the uniform probability measure on $\Lambda_0$ and put $\mu_0 = \nu_0^\Z$. For every $\beta \in \R$, there is a unique $\beta$-conformal measure for $\Lambda^{(\Z)} \actson X^{\Z}$ and this is given by $\mu_\beta = \mu_0 \times \mu_{1,\beta} \times \mu_{2,\beta}$. By Lemma \ref{lem.genconformal}, the KMS spectrum for the action $G \actson X^\Z$ and the $1$-cocycle $\Om$ thus consists of all $\beta \in \R$ such that
$$\frac{d(\theta_*\mu_\beta)}{d\mu_\beta}(x,y,z) = \exp(\beta \cdot \Om(a,(x,y,z)))$$
for $\mu_\be$-a.e.\ $(x,y,z) \in X^{\Z}$. By construction, we have
$$
\frac{d(\theta_*\mu_\beta)}{d\mu_\beta}(x,y,z)
= \begin{cases}\displaystyle 1 &\;\;\text{if $x \in Y_0$,}\\
|\Lambda_{0}|^{-1} \, \vphi_1(\beta)^{-1} \, H_1(y_0)^\beta &\;\;\text{if $x \in Y_1$,}\\
|\Lambda_{0}| \, \vphi_2(\beta)^{-1} \, H_2(z_0)^\be &\;\;\text{if $x \in Y_2$.}\end{cases}$$
The KMS spectrum is thus given by
$$K = \bigl\{\beta \in \R \bigm| \vphi_1(\beta) = |\Lambda_{0}|^{-1} \;\;\text{and}\;\; \vphi_2(\beta) = |\Lambda_{0}| \bigr\} \; .$$
This concludes the proof of the proposition.
\end{proof}

The action of $G$ in Proposition \ref{prop.cantorset} has most of the properties required in the second part of Theorem \ref{thm.all-closed-sets}, but it is not topologically free. In fact the action has a non-trivial isotropy group at every point of the space. To amend this, we need the following generic construction to obtain a topologically free action without changing the KMS spectrum.

\begin{lemma} \label{lemma.dummyaction}
Assume $\Gamma \actson Z$ is a free action of a countable group by homeomorphisms on a compact space $Z$ with a unique invariant probability measure $\mu$. Assume $\Gamma_{0} \subset \Gamma$ is a subgroup such that the restricted action $\Gamma_{0} \actson Z$ is minimal with unique invariant probability measure $\mu$.

Assume $G \actson X$ and $\pi: \Gamma \to G$ is a surjective group homomorphism with $\Gamma_{0} \subset \Ker \pi$. Let $\Om : G \times X \recht \R$ be a continuous $1$-cocycle. Define the action $\Gamma \actson Z \times X : g \cdot (z,x) = (g \cdot z, \pi(g) \cdot x)$. Then, $\Gamma \curvearrowright Z \times X$ is free.
\begin{enumlist}
\item If $G \actson X$ is minimal then $\Gamma \actson Z \times X$ is minimal.
\item The map $\Om_1 : \Gamma \times (Z \times X) \recht \R : \Om_1(g,(z,x)) = \Om(\pi(g),x)$ is a continuous $1$-cocycle.
\item The $e^{\be \cdot \Om_1}$-conformal probability measures are precisely the measures of the form $\mu \times \nu$ where $\nu$ is an $e^{\be \cdot \Om}$-conformal probability measure.
\item The KMS spectra of $(G \actson X,\Om)$ and $(\Gamma \actson Z \times X,\Om_1)$ coincide.
\end{enumlist}
\end{lemma}
\begin{proof}
1. Let $S \subset Z \times X$ be closed, nonempty and $\Gamma$-invariant. Since $\Gamma_0 \actson Z$ is minimal, for every $(x,y) \in S$, we get that $Z \times \{y\} \subset S$. Thus, $S = Z \times L$ where $L \subset X$ is closed. Since $S$ is $\Gamma$-invariant, we have that $L$ is $G$-invariant. Also, $L \neq \emptyset$. Thus $L = X$.

2. This is a direct computation which we leave to the reader.

3. If $\nu$ is $e^{\be \cdot \Om}$-conformal, it is immediate that $\mu \times \nu$ is $e^{\be \cdot \Om_1}$-conformal. If $\eta$ is an $e^{\beta \cdot \Om_{1}}$-conformal probability measure on $Z \times X$ it follows from the definition of $\Omega_{1}$ that $\eta$ is $\Gamma_{0}$-invariant. By Lemma \ref{lemma.prod-measure} then $\eta = \mu \times \nu$, where $\nu$ is a probability measure on $X$. Expressing that $\eta$ is $e^{ \be \cdot \Omega_{1}}$-conformal gives that $\nu$ is $e^{\be \cdot \Om}$-conformal. Since $4$ follows from $3$ this concludes the proof of the lemma.
\end{proof}

We will use the following action $\Gamma \actson Z$ satisfying all the assumptions in Lemma \ref{lemma.dummyaction}. We denote the $p$-adic integers as $\Z_p$.

\begin{lemma}\label{lem.dense-embedding}
Let $p$ be an odd prime and put $Z = \SL(2,\Z_{p})$. Define
$$a = \begin{pmatrix} 1 & 2 \\ 0 & 1 \end{pmatrix} \quad , \quad b = \begin{pmatrix} 1 & 0 \\ 2 & 1 \end{pmatrix} \; .$$
The elements $g_1 = a^4$, $g_2 = b^4$ and $h_n = (ab)^n \, ba \, (ab)^{-n}$, $n \in \Z$, are free, generating a copy of $\Gamma = \F_\infty$ inside $Z$, and the subgroup $\langle g_1,g_2 \rangle \subset Z$ is dense.
\end{lemma}

To prove Lemma \ref{lem.dense-embedding} consider the following easy result.

\begin{lemma}\label{lem.in-F2}
Let $\F_2$ be freely generated by $a,b \in \F_2$. Then the elements $g_1 = a^4$, $g_2 = b^4$ and $h_n = (ab)^n \, ba \, (ab)^{-n}$, $n \in \Z$, are free.
\end{lemma}
\begin{proof}
Put $k_1 = ab$ and $k_2 = ba$. Below we prove that $g_1,g_2,k_1,k_2$ are free. Since $h_n = k_1^n k_2 k_1^{-n}$, it then follows that the elements $(h_n)_{n \in \Z}$ are free inside $\langle k_1 , k_2 \rangle$. So the conclusion of the lemma holds.

It remains to prove that $g_1,g_2,k_1,k_2$ are free. Consider a nonempty word in the letters $(g_1^n)_{n \neq 0}$, $(g_2^m)_{m \neq 0}$, $k_1$, $k_1^{-1}$, $k_2$, $k_2^{-1}$, where a letter of the form $g_1^n$ is never followed by a letter of the same form, where a letter of the form $g_2^m$ is never followed by a letter of the same form, where $k_1$ is never followed or preceded by $k_1^{-1}$, and where $k_2$ is never followed or preceded by $k_2^{-1}$. We have to prove that this word defines a nontrivial element $g$ of $\F_2$. When reducing $g$ as a word in $a,b$, the only powers of $a$ that appear are:
$$a^{\pm 1} \;\; , \;\; a^{\pm 2} \;\; , \;\; a^{4n} \;\;, \;\;  a^{4n \pm 1} \;\; , \;\; a^{\pm 1} \, a^{4n} \, a^{\pm 1} = a^{4n + i} \; , \; i \in \{0,2,-2\} \; ,$$
with $n \in \Z \setminus \{0\}$. None of these powers are $0$. A similar thing happens with the powers of $b$. It follows that $g \neq e$.
\end{proof}

\begin{proof}[{Proof of Lemma \ref{lem.dense-embedding}}]
The elements $a,b$ generate a free subgroup of $\SL(2,\Z)$ and $\SL(2,\Z) \subset Z$. So, $a,b$ are free elements in $Z$. It follows from Lemma \ref{lem.in-F2} that $g_1$, $g_2$, $(h_n)_{n \in \Z}$ are free.

It remains to prove that $\langle g_1,g_2 \rangle$ is a dense subgroup of $Z$. Denote by $L$ the closure of $\langle g_1,g_2 \rangle$. Note that for $n \in \Z$
$$g_1^n = \begin{pmatrix} 1 & 8n \\ 0 & 1 \end{pmatrix} \; .$$
Since $p$ is odd, we have that $8 \Z$ is dense in $\Z_{p}$. It follows that $\bigl(\begin{smallmatrix} 1 & 1 \\ 0 & 1 \end{smallmatrix}\bigr) \in L$. We similarly find that $\bigl(\begin{smallmatrix} 1 & 0 \\ 1 & 1 \end{smallmatrix}\bigr) \in L$. So, $\SL(2,\Z) \subset L$. Since $\SL(2,\Z)$ is a dense subgroup of $Z$, we conclude that $L = Z$.
\end{proof}

We have now collected all results to prove the second part of Theorem \ref{thm.all-closed-sets}.

\begin{proof}[Proof of \ref{thma2} in Theorem \ref{thm.all-closed-sets}]
Fix any nontrivial finite group $\Lambda_0$. Put $\Lambda_1 = \Lambda_2 = \Lambda_0^{(\N)}$ and consider the translation action of $\Lambda_i$ on $X_i = \Lambda_0^\N$, for $i=1,2$. We put $\Lambda = \Lambda_0 \times \Lambda_1 \times \Lambda_2$ and we consider the action of $G = \Lambda^{(\Z)} \ast \Z$ on the Cantor set $Y = X^{\Z}$ given by Proposition \ref{prop.cantorset}. Let $\F_{\infty}$ be the free group with generators $\{a_{n}\}_{n=1}^{\infty}$. Choose a surjective group homomorphism $\pi:\F_{\infty} \to G$ with $\langle a_{1}, a_{2} \rangle \subset \Ker \pi$. Consider the dense embedding of $\F_\infty$ into $Z = \SL(2,\Z_p)$ given by Lemma \ref{lem.dense-embedding}, where we identify $a_i = g_i$ for $i=1,2$ and $a_n = h_n$ for $n \geq 3$. We let $\F_\infty$ act on $Z$ by translation.

By Lemma \ref{lem.dense-embedding}, the action $\F_\infty \actson Z$ satisfies all the assumptions of Lemma \ref{lemma.dummyaction} w.r.t.\ the subgroup $\langle a_1,a_2 \rangle$ and the Haar measure $\mu$ on $Z$.

We claim that the diagonal action $\F_\infty \actson Z \times Y$ satisfies all the conclusions of the second part of Theorem \ref{thm.all-closed-sets}. Note that the action is free and minimal. Let $K \subset \R$ be a closed subset with $0 \not\in K$. By Lemma \ref{lemma.realizablefractions}, we can choose functions $\vphi_i : \R \to (0,+\infty)$ that are realizable for $\Lambda_i \actson X_i$ such that
$$\{\beta \in \R \mid  \vphi_1(\beta) = |\Lambda_0|^{-1}\} = K = \{\beta \in \R \mid \vphi_2(\beta) = |\Lambda_0| \} \; .$$

By Proposition \ref{prop.cantorset} there exists a continuous $1$-cocycle $\Omega: G \times Y \to \R$ with KMS spectrum $K$ and unique $\beta$-conformal measure for each $\beta \in K$. By Lemma \ref{lemma.dummyaction}, we also find a continuous $1$-cocycle for the action $\F_\infty \actson Z \times Y$ with KMS spectrum $K$ and unique $\beta$-conformal measure for every $\beta \in K$. Since this action is free, the result follows from Lemma \ref{lemma.cstaralgeba}.
\end{proof}

\section{Variation in the KMS state simplices} \label{section5}

Theorem \ref{thm.all-closed-sets} proves that $1$-parameter groups arising from continuous $1$-cocycles gives rise to the most extreme possible variation of KMS spectra. There is however only one $\beta$-KMS state for each $\beta$ in the spectrum. In this section we will strengthen Theorem \ref{thm.all-closed-sets} by proving that one can obtain an extreme variation in the size of the $\beta$-KMS simplices for varying $\beta$ in the KMS spectrum. We define $S_{\beta}$ to be the set of $\beta$-KMS states for a given $1$-parameter group on a unital C$^*$-algebra, and recall that $S_{\beta}$ is a simplex, c.f. Theorem 5.3.30 in \cite{BR}. We will prove the following result.

\begin{theorem}\label{thm.variation-of-simplex}
Let $(\Lambda_n)_{n \in \N}$ be a sequence of nontrivial finite groups. Put $\Lambda = \bigoplus_{n \in \N} \Lambda_n$ and consider the product action of $\Lambda$ on $X = \prod_{n\in \N} \Lambda_n$ by translation. Consider also the product action of $H=(\Z / 2\Z)^{(\N)}$ on $X_{0}=(\Z / 2\Z)^{\N}$ by translation.
\begin{enumlist}
\item\label{thmvar1} Consider the natural action of the wreath product $\Gamma = \Lambda \wr \Z = \Lambda^{(\Z)} \rtimes \Z$ on $X^\Z$. The minimal and topologically free action
$$
\Gamma\times H \actson X^{\Z}\times X_{0}
$$
has the following universality property: for every closed subset $K \subset \R$ with $0 \in K$, there exists a continuous $1$-cocycle $(\Gamma \times H) \times (X^{\Z}\times X_{0}) \recht \R$ whose KMS spectrum equals $K$, and such that for $\beta \in K\setminus\{0\}$ the set $S_{\beta}$ is an infinite dimensional Bauer simplex with $S_{\beta}$ not affinely homeomorphic to $S_{\beta'}$ for $\beta'\neq \beta$.
\item \label{thmvar2}$\F_{\infty} $ admits a minimal and free action on the Cantor set $Y$ with the following universality property: for every closed subset $K \subset \R$ with $0 \not\in K$, there exists a continuous $1$-cocycle $\F_{\infty} \times Y \recht \R$ whose KMS spectrum equals $K$, and for $\beta \in K$ then $S_{\beta}$ is an infinite dimensional Bauer simplex with $S_{\beta}$ not affinely homeomorphic to $S_{\beta'}$ for $\beta'\neq \beta$.
\end{enumlist}
\end{theorem}

To prove Theorem \ref{thm.variation-of-simplex} we will combine Theorem \ref{thm.all-closed-sets} with Theorem 1.1 in \cite{T2}. To do this we need the following interpretation of Theorem 1.1 in \cite{T2}.

\begin{lemma} \label{lemma.CAR-variation}
Consider the product action of $H=(\Z / 2\Z)^{(\N)}$ on $X_{0}=(\Z / 2\Z)^{\N}$ by translation. There is a continuous $1$-cocycle $\Omega: H \times X_{0}  \to \R$ such that the corresponding $1$-parameter group on
$$
C\big(X_{0}\big) \rtimes_{r} H
$$
satisfies that for all $\beta \neq 0$ the simplex $S_{\beta}$ of $\beta$-KMS states is an infinite dimensional Bauer simplex such that $S_{\beta}$ is not affinely homeomorphic to $S_{\beta '}$ for $\beta \neq \beta'$.
\end{lemma}

\begin{proof}
Arguing as in the proof of Theorem \ref{thmuhfalg} we see that $C(X_{0}) \rtimes_{r} H$ is naturally isomorphic to the CAR algebra $\mathcal{C}$, such that $C(X_{0})$ is mapped to the diagonal of $\mathcal{C}$. The map
$$
(x_{i})_{i\in \N} \mapsto \left ( (x_{2i})_{i\in \N} \; , \; (x_{2i-1})_{i\in \N} \right )
$$
gives rise to a homeomorphism $X_{0} \to X_{0} \times X_{0}$ and a group isomorphism $H \to H\times H$ which are equivariant. This gives rise to the first isomorphism in
$$
C\big(X_{0}\big) \rtimes_{r} H
\simeq
C\big(X_{0} \times X_{0}\big) \rtimes_{r} \big(H \times H \big)
\simeq
C\big(X_{0} \big) \rtimes_{r} H \otimes C\big(X_{0} \big) \rtimes_{r} H
$$
where the second isomorphism in particular sends $C(X_{0}) \otimes C(X_{0})$ to a dense subalgebra of $C(X_{0} \times X_{0})$. Combined, we get an isomorphism
\begin{equation} \label{eq.CAR-tensor}
\mathcal{C} \otimes  \mathcal{C} \simeq
\mathcal{C}
\end{equation}
which maps the span of elements of the form $x_{1} \otimes x_{2}$ with $x_{1}$ and $x_{2}$ in the diagonal of $\mathcal{C}$ onto a dense subset of the diagonal of $\mathcal{C}$.

We now want to use the results of \cite{T2}, which constructs a $1$-parameter group on $\mathcal{C}$ with an extreme variation in $\beta$-KMS simplices. Since it is important that the $1$-parameter group we construct arises from some continuous $1$-cocycle $\Omega: H \times X_{0}  \to \R$, we will in the following account for the construction of this $1$-parameter group. We will argue that it fixes the diagonal, which by Lemma \ref{lem.diagonalfixing} implies that it arises from a continuous $1$-cocycle. Going through the construction in \cite{T2}, and using that the isomorphism \eqref{eq.CAR-tensor} respects the diagonals, one notices that the $1$-parameter group $\alpha$ constructed in Theorem 5.5 in \cite{T2} indeed fixes the diagonal. The proof of Theorem 1.1 in \cite{T2} now finishes the proof of the Lemma.
\end{proof}

We can now use Lemma \ref{lemma.CAR-variation} to obtain a \emph{proof of Theorem \ref{thm.variation-of-simplex}}.  We construct a $1$-cocycle $\Omega_{2}: H  \times X_{0}  \to \R$ using Lemma \ref{lemma.CAR-variation}, and remark that the set of $\beta$-KMS states for this $1$-parameter group is homeomorphic with the set of $\beta$-conformal measures for all $\beta \in \R$.

For \ref{thmvar1} we consider the action $\Gamma \actson Y$ as in Theorem \ref{thm.all-closed-sets-amenable}. If $K\subset \R$ is a closed set with $0\in K$ then there exists a $1$-cocycle on $\Gamma \times Y$ with KMS spectrum $K$. By Lemma \ref{lemma.prod-measure} we obtain a $1$-cocycle $\Omega$ on $(\Gamma \times H)\times(Y\times X_{0})$ such that there exists $\beta$-conformal measures on $Y\times X_{0}$ if and only if $\beta \in K$, and when $\beta \in K$ the map $\mu \to \nu\times \mu$ is an affine homeomorphism from the $e^{\beta \cdot \Om_{2}}$-conformal measures to the set of $e^{\beta \cdot \Omega}$-conformal measures. Since all these measures are essentially free for the action of $\Gamma \times H$,  \ref{thmvar1} follows from Lemma \ref{lemma.cstaralgeba}.

Proceeding as in the proof of \ref{thmvar1}, but now using the second part of Theorem \ref{thm.all-closed-sets}, we obtain a minimal and free action of  $\F_{\infty} \times  H$ on the Cantor set $Y$ with the following universality property: for every closed subset $K \subset \R$ with $0 \not\in K$, there exists a continuous $1$-cocycle $(\F_{\infty}\times H) \times Y \recht \R$ whose KMS spectrum equals $K$, and for $\beta \in K$ then $S_{\beta}$ is an infinite dimensional Bauer simplex with $S_{\beta}$ not affinely homeomorphic to $S_{\beta'}$ for $\beta'\neq \beta$. Using Lemma \ref{lemma.dummyaction} with $\Gamma = \F_{\infty}$ acting on $Z = \SL(2,\Z_{p})$ and $G=\F_{\infty} \times  H$ then proves the statement in \ref{thmvar2}.\qed


\begin{thebibliography}{WWWW} 

\bibitem[aHLRS15]{aHLRS} A. an Huef, M. Laca, I. Raeburn and A. Sims, {\em KMS states on the C$^{*}$-algebras of reducible graphs.} Ergod. Theory Dyn. Syst. {\bf 35} (2015), 2535-2558.

\bibitem[ALN20]{ALN20} Z. Afsar, N. Larsen and S. Neshveyev, {\em  KMS States on Nica-Toeplitz C$^{*}$-algebras.} Comm. Math. Phys. {\bf 378} (2020), 1875-1929.

\bibitem[BD91]{BD89} D. Boivin and Y. Derriennic, {\em The ergodic theorem for additive cocycles of $\Z^d$ or $\R^d$.} Ergod. Theory Dyn. Syst. {\bf 11} (1991), 19-39.

\bibitem[BEH80]{BEH} O. Bratteli, G.A. Elliott and R.H Herman, {\em On the possible temperatures of a dynamical system.} Comm. Math. Phys. {\bf 74} (1980), 281-295.

\bibitem[BEK80]{BEK} O. Bratteli, G. Elliott and A. Kishimoto, {\em The temperature state space of a dynamical system I.} J. Yokohama Univ. {\bf 28} (1980), 125-167.

\bibitem[BR79]{BR} O. Bratteli and D.W. Robinson, {\em Operator Algebras and Quantum Statistical Mechanics I, II.} Texts and Monographs in Physics, Springer Verlag, New York, Heidelberg, Berlin, 1979, 1981.

\bibitem[CT19]{CT} J. Christensen and K. Thomsen, {\em KMS states on the crossed product C$^*$-algebra of a homeomorphism.} Ergod. Theory Dyn. Syst. (to appear), arXiv:1912.06069.

\bibitem[DU91]{DU} M. Denker and M. Urbanski, {\em On the existence of conformal measures.} Trans. Amer. Math. Soc. {\bf 328} (1991), 563-587.

\bibitem[KR06]{KR} A. Kumjian and J. Renault, {\em KMS states on $C^{*}$-algebras associated to expansive maps.} Proc. Amer. Math. Soc.  {\bf 134} (2006), 2067-2078.

\bibitem[Nes11]{N} S. Neshveyev, {\em KMS states on the C$^*$-algebras of non-principal groupoids.} J. Operator Theory {\bf 70} (2011), 513-530.

\bibitem[OP78]{OP} D. Olesen and G. K. Pedersen, {\em Some C$^{*}$-dynamical systems with a single KMS state.} Math. Scand. {\bf 42} (1978), 111-118.

\bibitem[Ren80]{R} J. Renault, {\em A Groupoid Approach to C$^*$-algebras.} Lecture Notes in Mathematics {\bf 793}, Springer Verlag, Berlin, Heidelberg, New York, 1980.

\bibitem[Tho17]{T1} K. Thomsen, {\em KMS weights on graph $C^*$-algebras.} Adv. Math. {\bf 309} (2017), 334-391.

\bibitem[Tho20]{T2} K. Thomsen {\em Phase transition in the CAR algebra.} Adv. Math. {\bf 372} (2020), art.\ id.\ 107312, 27 pp.

\bibitem[Tho21]{T3} K. Thomsen {\em On the possible temperatures for flows on a UHF algebra.} Preprint, arXiv:2012.03306.
\end{thebibliography}
\end{document}